\documentclass[12pt]{amsart}

\usepackage[utf8]{inputenc}
\usepackage[T1]{fontenc}
\usepackage[british]{babel}

\usepackage{textcomp}
\usepackage{scrhack}
\usepackage{xspace}
\usepackage{mparhack}
\usepackage{fixltx2e}

\usepackage{amsmath}
\usepackage[warning, all]{onlyamsmath}\AtBeginDocument{\catcode`\$=3}
\usepackage{amssymb}
\usepackage{stmaryrd}
\usepackage{dsfont}
\usepackage{mathtools}
\usepackage{mathdots}
\usepackage{gensymb}
\allowdisplaybreaks

\usepackage{graphicx}
\usepackage{subfig}
\usepackage{scalerel}
\usepackage{csquotes}
\usepackage{url}
\usepackage{trimclip, adjustbox}
\usepackage{tabularx}

\usepackage{caption}
\usepackage{tikz}
\usetikzlibrary{positioning, arrows, patterns}

\usepackage[pdftex, hyperfootnotes = false, pdfpagelabels]{hyperref}
\usepackage[noabbrev, capitalise]{cleveref} 

\newtheorem{theorem}{Theorem}[section]
\newtheorem{main-theorem}[theorem]{Main Theorem}
\newtheorem{corollary}[theorem]{Corollary}
\newtheorem{lemma}[theorem]{Lemma}

\theoremstyle{definition}
\newtheorem{definition}[theorem]{Definition}
\newtheorem{remark}[theorem]{Remark}
\newtheorem{example}[theorem]{Example}

\newtheoremstyle{proof-sketch}%
  {}{}%
  {}{}%
  {\itshape}{.}%
  { }%
  {\thmname{#1}\thmnumber{ #2}\thmnote{ (#3)}}
\theoremstyle{proof-sketch}
\newtheorem*{proof-sketch}{Proof Sketch}
\newtheorem*{usage-note}{Usage Note}

\makeatletter
\newcommand{\proofpart}[1]{%
  \par%
  \addvspace{\medskipamount}%
  \noindent\emph{#1.}\par\nobreak%
  \addvspace{\smallskipamount}%
  \@afterheading%
}
\makeatother

\urldef{\mails}\path|simon.wacker@kit.edu|

\newcommand*{\define}[1]{\emph{#1}}

\providecommand\implies{\DOTSB\;\Longrightarrow\;}

\DeclarePairedDelimiter\parens{\lparen}{\rparen}

\DeclarePairedDelimiter\absoluteValueOf{\lvert}{\rvert}

\DeclarePairedDelimiterX\inner[2]{\langle}{\rangle}{#1,#2} 
\DeclarePairedDelimiter\setOf{\{}{\}} 
\DeclarePairedDelimiter\net{\{}{\}}
\DeclarePairedDelimiter\family{\{}{\}}
\DeclarePairedDelimiter\sequence{\lparen}{\rparen}

\DeclarePairedDelimiter\gen{\langle}{\rangle}

\mathchardef\breakingcomma\mathcode`\,
{\catcode`,=\active
  \gdef,{\breakingcomma\discretionary{}{}{}}
}
\newcommand*\ntuple[1]{\lparen\mathcode`\,=\string"8000 #1\rparen}

\newcommand*{\N}{\mathbb{N}}
\newcommand*{\Z}{{\mathbb{Z}}}
\newcommand*{\R}{\mathbb{R}}

\newcommand*{\ball}{\mathbb{B}}
\newcommand*{\sphere}{\mathbb{S}}

\DeclareMathOperator{\Exists}{\exists}
\DeclareMathOperator{\ForEach}{\forall}
\newcommand*\Holds{:}
\newcommand*\SuchThat{:}

\newcommand*\actsOnPoint{\triangleright}
\newcommand*\actsOnMap{\mathbin{\protect\scalerel*{\blacktriangleright}{\triangleright}}}
\newcommand*\isSemiActedUponBy{\mathbin{\protect\scalerel*{\trianglelefteqslant}{\rhd}}} 
\newcommand*\actsByItsCoordinateOn{\mathbin{\protect\scalerel*{\blacktriangleleft}{\triangleleft}}} 

\DeclareMathOperator{\symmetricGroupOf}{Sym}

\makeatletter
\def\moverlay{\mathpalette\mov@rlay}
\def\mov@rlay#1#2{\leavevmode\vtop{%
   \baselineskip\z@skip \lineskiplimit-\maxdimen
   \ialign{\hfil$\m@th#1##$\hfil\cr#2\crcr}}}
\newcommand{\charfusion}[3][\mathord]{
    #1{\ifx#1\mathop\vphantom{#2}\fi
        \mathpalette\mov@rlay{#2\cr#3}
      }
    \ifx#1\mathop\expandafter\displaylimits\fi}
\makeatother
\newcommand{\disjointUnionWith}{\charfusion[\mathbin]{\cup}{\cdot}}

\newcommand*\boundaryOf{\partial}
\DeclareMathOperator{\entropyOf}{ent}
\DeclareMathOperator{\powerSetOf}{\mathcal{P}}
\DeclareMathOperator{\domainOf}{dom}
\DeclareMathOperator{\differenceOf}{diff}
\DeclareMathOperator{\cylinder}{Cyl}
\newcommand*\occursIn{\sqsubseteq}
\newcommand*\semiOccursIn{\mathrel{\ooalign{$\occursIn$\cr\hidewidth\raise.225ex\hbox{$\circ\mkern.5mu$}\cr}}} 

\newcommand*\suchThat{\mid} 
\newcommand*\from{\colon} 

\newcommand*\modulo{\slash}
\newcommand*{\restrictedTo}{\mathord{\upharpoonright}}

\newcommand*{\blank}{\mathord{\_}} 

\newcommand*{\graffito}[1]{}
\newcommand*{\mathnote}[1]{}
\renewcommand*{\index}[1]{}

\begin{document}
    \title[The Moore and the Myhill Property]{The Moore and the Myhill Property For Strongly Irreducible Subshifts Of Finite Type Over Group Sets}
    \author{Simon Wacker}
    \address{%
      Simon Wacker\\
      Department of Informatics\\
      Karlsruhe Institute of Technology\\
      Am Fasanengarten 5\\
      76131 Karlsruhe\\
      Germany
    }
    \email{simon.wacker@kit.edu}


    \maketitle

    \begin{abstract}
      We prove the Moore and the Myhill property for strongly irreducible subshifts over right amenable and finitely right generated left homogeneous spaces with finite stabilisers. Both properties together mean that the global transition function of each big-cellular automaton with finite set of states and finite neighbourhood over such a subshift is surjective if and only if it is pre-injective. This statement is known as Garden of Eden theorem. Pre-Injectivity means that two global configurations that differ at most on a finite subset and have the same image under the global transition function must be identical.
%
    \end{abstract}



  The notion of amenability for groups was introduced by John von Neumann in 1929. It generalises the notion of finiteness. A group $G$ is \emph{left} or \emph{right amenable} if there is a finitely additive probability measure on $\powerSetOf(G)$ that is invariant under left and right multiplication respectively. Groups are left amenable if and only if they are right amenable. A group is \emph{amenable} if it is left or right amenable.

  The definitions of left and right amenability generalise to left and right group sets respectively. A left group set $\ntuple{M, G, \actsOnPoint}$ is \emph{left amenable} if there is a finitely additive probability measure on $\powerSetOf(M)$ that is invariant under $\actsOnPoint$. There is in general no natural action on the right that is to a left group action what right multiplication is to left group multiplication. Therefore, for a left group set there is no natural notion of right amenability.

  A transitive left group action $\actsOnPoint$ of $G$ on $M$ induces, for each element $m_0 \in M$ and each family $\family{g_{m_0, m}}_{m \in M}$ of elements in $G$ such that, for each point $m \in M$, we have $g_{m_0, m} \actsOnPoint m_0 = m$, a right quotient set semi-action $\isSemiActedUponBy$ of $G \modulo G_0$ on $M$ with defect $G_0$ given by $m \isSemiActedUponBy g G_0 = g_{m_0, m} g g_{m_0, m}^{-1} \actsOnPoint m$, where $G_0$ is the stabiliser of $m_0$ under $\actsOnPoint$. Each of these right semi-actions is to the left group action what right multiplication is to left group multiplication. They occur in the definition of global transition functions of semi-cellular automata over left homogeneous spaces as defined in \cite{wacker:automata:2016}. A \emph{cell space} is a left group set together with choices of $m_0$ and $\family{g_{m_0, m}}_{m \in M}$.

  A cell space $\mathcal{R}$ is \emph{right amenable} if there is a finitely additive probability measure on $\powerSetOf(M)$ that is semi-invariant under $\isSemiActedUponBy$. For example cell spaces with finite sets of cells, abelian groups, and finitely right generated cell spaces with finite stabilisers of sub-exponential growth are right amenable, in particular, quotients of finitely generated groups of sub-exponential growth by finite subgroups acted on by left multiplication. A net of non-empty and finite subsets of $M$ is a \emph{right Følner net} if, broadly speaking, these subsets are asymptotically invariant under $\isSemiActedUponBy$. A finite subset $E$ of $G \modulo G_0$ and two partitions $\family{A_e}_{e \in E}$ and $\family{B_e}_{e \in E}$ of $M$ constitute a \emph{right paradoxical decomposition} if the map $\blank \isSemiActedUponBy e$ is injective on $A_e$ and $B_e$, and the family $\family{(A_e \isSemiActedUponBy e) \disjointUnionWith (B_e \isSemiActedUponBy e)}_{e \in E}$ is a partition of $M$. The Tarski-Følner theorem states that right amenability, the existence of right Følner nets, and the non-existence of right paradoxical decompositions are equivalent. We prove it in \cite{wacker:amenable:2016} for cell spaces with finite stabilisers.

  A cell space $\mathcal{R}$ is \emph{finitely right generated} if there is a finite subset $S$ of $G \modulo G_0$ such that, for each cell $m \in M$, there is a family $\family{s_i}_{i \in \setOf{1,2,\dotsc,k}}$ of elements in $S \cup S^{-1}$ such that $m = (((m_0 \isSemiActedUponBy s_1) \isSemiActedUponBy s_2) \isSemiActedUponBy \dotsb) \isSemiActedUponBy s_k$. The finite right generating set $S$ induces the \emph{$S$-Cayley graph} structure on $M$: For each cell $m \in M$ and each generator $s \in S$, there is an edge from $m$ to $m \isSemiActedUponBy s$. The length of the shortest path between two points of $M$ yields the \emph{$S$-metric}. 

  A subset $X$ of $Q^M$, where $Q$ is a finite set, is a \emph{shift space of finite type} if it is generated by a finite set of forbidden blocks or, equivalently, if it is shift-invariant and compact. 
  Such a space $X$ is \emph{strongly irreducible} if each pair of finite patterns that are allowed in $X$ and at least some fixed positive integer apart, are embedded in a point of $X$. A map $\Delta$ from a shift space $X$ to a shift space $Y$ is \emph{local} if the state $\Delta(x)(m)$ is uniformly and locally determined in $m$, in other words, if the map $\Delta$ is the restriction of the global transition function of a big-cellular automaton with finite neighbourhood to the domain $X$ and the codomain $Y$. 

  For a right amenable and finitely right generated cell space with finite stabilisers we may choose a right Følner net $\mathcal{F} = \family{F_i}_{i \in I}$. The \define{entropy} of a subset $X$ of $Q^M$ with respect to $\mathcal{F}$, where $Q$ is a finite set, is, broadly speaking, the asymptotic growth rate of the number of finite patterns with domain $F_i$ that occur in $X$. For non-negative integers $\theta$, $\kappa$, and $\theta'$, a \emph{$\ntuple{\theta, \kappa, \theta'}$-tiling} is a subset $T$ of $M$ such that $\family{\ball(t, \theta)}_{t \in T}$ is pairwise at least $\kappa + 1$ apart and $\family{\ball(t, \theta')}_{t \in T}$ is a cover of $M$. If for each point $t \in T$ not all patterns with domain $\ball(t, \theta)$ occur in a subset of $Q^M$, then that subset does not have maximal entropy. 

  A local map from a non-empty strongly irreducible shift space of finite type to a strongly irreducible shift space with the same entropy over a right amenable and finitely right generated cell space with finite stabilisers is surjective if and only if its image has maximal entropy
  and its image has maximal entropy if and only if it is pre-injective.
  This establishes the Garden of Eden theorem,
  which states that a local map as above is surjective if and only if it is pre-injective. This answers a question posed by Sébastien Moriceau at the end of his paper \enquote{Cellular Automata on a $G$-Set}\cite{moriceau:2011}. And it follows that strongly irreducible shift spaces of finite type over right amenable and finitely right generated cell spaces have the Moore and the Myhill property.

  The Garden of Eden theorem for cellular automata over $\Z^2$ is a famous theorem by Edward Forrest Moore and John R. Myhill from 1962 and 1963, which was proved in their papers \enquote{Machine models of self-reproduction}\cite{moore:1962} and \enquote{The converse of Moore's Garden-of-Eden theorem}\cite{myhill:1963}. That theorem also holds for cellular automata over amenable finitely generated groups, which was proved by Tullio Ceccherini-Silberstein, Antonio Machi, and Fabio Scarabotti in their paper \enquote{Amenable groups and cellular automata}\cite{ceccherini-silberstein:machi:scarabotti:1999}. It even holds for such automata on strongly irreducible shifts of finite type, which was proved by Francesca Fiorenzi in her paper \enquote{Cellular automata and strongly irreducible shifts of finite type}\cite{fiorenzi:2003}. The present paper generalises results from and is greatly inspired by Francesca Fiorenzi's paper.

  In \cref{sec:full-shifts} we introduce full shifts, patterns, blocks, and shift-invariance. In \cref{sec:shift-spaces} we introduce shift spaces or subshifts (of finite type), strong irreducibility, bounded propagation, local maps, conjugacies, and the Moore and the Myhill property. In \cref{sec:tilings} we introduce tilings, prove their existence, and relate them to entropies. And in \cref{sec:gardens-of-eden} we prove the Garden of Eden theorem, from which we deduce that both the Moore and the Myhill property hold.

  \subsubsection*{Preliminary Notions.} A \define{left group set} is a triple $\ntuple{M, G, \actsOnPoint}$, where $M$ is a set, $G$ is a group, and $\actsOnPoint$ is a map from $G \times M$ to $M$, called \define{left group action of $G$ on $M$}, such that $G \to \symmetricGroupOf(M)$, $g \mapsto [g \actsOnPoint \blank]$, is a group homomorphism. The action $\actsOnPoint$ is \define{transitive} if $M$ is non-empty and for each $m \in M$ the map $\blank \actsOnPoint m$ is surjective; and \define{free} if for each $m \in M$ the map $\blank \actsOnPoint m$ is injective. For each $m \in M$, the set $G \actsOnPoint m$ is the \define{orbit of $m$}, the set $G_m = (\blank \actsOnPoint m)^{-1}(m)$ is the \define{stabiliser of $m$}, and, for each $m' \in M$, the set $G_{m, m'} = (\blank \actsOnPoint m)^{-1}(m')$ is the \define{transporter of $m$ to $m'$}.

  A \define{left homogeneous space} is a left group set $\mathcal{M} = \ntuple{M, G, \actsOnPoint}$ such that $\actsOnPoint$ is transitive. A \define{coordinate system for $\mathcal{M}$} is a tuple $\mathcal{K} = \ntuple{m_0, \family{g_{m_0, m}}_{m \in M}}$, where $m_0 \in M$ and for each $m \in M$ we have $g_{m_0, m} \actsOnPoint m_0 = m$. The stabiliser $G_{m_0}$ is denoted by $G_0$. The tuple $\mathcal{R} = \ntuple{\mathcal{M}, \mathcal{K}}$ is a \define{cell space}. The map $\isSemiActedUponBy \from M \times G \modulo G_0 \to M$, $(m, g G_0) \mapsto g_{m_0, m} g g_{m_0, m}^{-1} \actsOnPoint m\ (= g_{m_0, m} g \actsOnPoint m_0)$ is a \define{right semi-action of $G \modulo G_0$ on $M$ with defect $G_0$}, which means that
  \begin{gather*}
    \ForEach m \in M \Holds m \isSemiActedUponBy G_0 = m,\\
    \ForEach m \in M \ForEach g \in G \Exists g_0 \in G_0 \SuchThat \ForEach \mathfrak{g}' \in G \modulo G_0 \Holds
          m \isSemiActedUponBy g \cdot \mathfrak{g}' = (m \isSemiActedUponBy g G_0) \isSemiActedUponBy g_0 \cdot \mathfrak{g}'.
  \end{gather*}
  It is \define{transitive}, which means that the set $M$ is non-empty and for each $m \in M$ the map $m \isSemiActedUponBy \blank$ is surjective; and \define{free}, which means that for each $m \in M$ the map $m \isSemiActedUponBy \blank$ is injective; and \define{semi-commutes with $\actsOnPoint$}, which means that
  \begin{equation*}
    \ForEach m \in M \ForEach g \in G \Exists g_0 \in G_0 \SuchThat \ForEach \mathfrak{g}' \in G \modulo G_0 \Holds
          (g \actsOnPoint m) \isSemiActedUponBy \mathfrak{g}' = g \actsOnPoint (m \isSemiActedUponBy g_0 \cdot \mathfrak{g}').
  \end{equation*}
  The maps $\iota \from M \to G \modulo G_0$, $m \mapsto G_{m_0, m}$, and $m_0 \isSemiActedUponBy \blank$ are inverse to each other. Under the identification of $M$ with $G \modulo G_0$ by either of these maps, we have $\isSemiActedUponBy \from (m, \mathfrak{g}) \mapsto g_{m_0, m} \actsOnPoint \mathfrak{g}$. (See \cite{wacker:automata:2016}.)

  A left homogeneous space $\mathcal{M}$ is \define{right amenable} if there is a coordinate system $\mathcal{K}$ for $\mathcal{M}$ and there is a finitely additive probability measure $\mu$ on $M$ such that 
  \begin{equation*}
    \ForEach \mathfrak{g} \in G \modulo G_0 \ForEach A \subseteq M \Holds \parens[\big]{(\blank \isSemiActedUponBy \mathfrak{g})\restrictedTo_A \text{ injective} \implies \mu(A \isSemiActedUponBy \mathfrak{g}) = \mu(A)},
  \end{equation*}
  in which case the cell space $\mathcal{R} = \ntuple{\mathcal{M}, \mathcal{K}}$ is called \define{right amenable}. When the stabiliser $G_0$ is finite, that is the case if and only if there is a \define{right Følner net in $\mathcal{R}$ indexed by $(I, \leq)$}, which is a net $\net{F_i}_{i \in I}$ in $\setOf{F \subseteq M \suchThat F \neq \emptyset, F \text{ finite}}$ such that
  \begin{equation*} 
    \ForEach \rho \in \N_0 \Holds \lim_{i \in I} \frac{\absoluteValueOf{\boundaryOf_\rho F_i}}{\absoluteValueOf{F_i}} = 0.
  \end{equation*}
  If a net is a right Følner net for one coordinate system, then it is a right Følner net for each coordinate system. In particular, a left homogeneous space $\mathcal{M}$ with finite stabilisers is right amenable if and only if, for each coordinate system $\mathcal{K}$ for $\mathcal{M}$, the cell space $\ntuple{\mathcal{M}, \mathcal{K}}$ is right amenable. (See \cite{wacker:amenable:2016,wacker:growth:2017}.)

  A left homogeneous space $\mathcal{M}$ is \define{finitely right generated} if there is a coordinate system $\mathcal{K}$ for $\mathcal{M}$ and
  there is a finite subset $S$ of $G \modulo G_0$ such that $G_0 \cdot S \subseteq S$ and, for each $m \in M$, there is a $k \in \N_0$ and there is a $\family{s_i}_{i \in \setOf{1,2,\dotsc,k}} \subseteq S \cup S^{-1}$, where $S^{-1} = \setOf{g^{-1} G_0 \suchThat s \in S, g \in s}$, such that
  \begin{equation*}
    m = \parens[\Big]{\parens[\big]{(m_0 \isSemiActedUponBy s_1) \isSemiActedUponBy s_2} \dotsb} \isSemiActedUponBy s_k.
  \end{equation*}
  in which case the cell space $\mathcal{R} = \ntuple{\mathcal{M}, \mathcal{K}}$ is called \define{finitely right generated}. The left homogeneous space $\mathcal{M}$ is finitely right generated if and only if, for each coordinate system $\mathcal{K}$ for $\mathcal{M}$, the cell space $\ntuple{\mathcal{M}, \mathcal{K}}$ is finitely right generated.
  The right generating set $S$ is \define{symmetric} if $S^{-1} \subseteq S$. The $S$-edge-labelled directed multigraph $\ntuple{M, E, \sigma, \tau, \lambda}$, where $E = \setOf{(m, s, m \isSemiActedUponBy s) \suchThat m \in M, s \in S}$ and $\sigma \from E \to M$, $\lambda \from E \to S$, and $\tau \from E \to M$ are the projections to the first, second, and third component respectively, is the \define{$S$-Cayley graph}. The distance $d_S$ on that graph is the \define{$S$-metric} and the map $\absoluteValueOf{\blank}_S = d_S(m_0, \blank)$ is the \define{$S$-length}. For each $m \in M$ and each $\rho \in \N_0$, the sets
  \begin{align*}
    \ball_S(m, \rho) &= \setOf{m' \in M \suchThat d_S(m, m') \leq \rho},\\
    \sphere_S(m, \rho) &= \setOf{m' \in M \suchThat d_S(m, m') = \rho}
  \end{align*}
  are the \define{ball/sphere of radius $\rho$ centred at $m$}, the ball $\ball_S(m_0, \rho)$ is denoted by $\ball_S(\rho)$, and the sphere $\sphere_S(m_0, \rho)$ by $\sphere_S(\rho)$. For each $A \subseteq M$, each $\theta \in \N_0$, the set $A^{-\theta} = \setOf{m \in A \suchThat \ball_S(m, \theta) \subseteq A}$ is the \define{$\theta$-interior of $A$}, the set $A^{+\theta} = \setOf{m \in M \suchThat \ball_S(m, \theta) \cap A \neq \emptyset}$ is the \define{$\theta$-closure of $A$}, the set $\boundaryOf_\theta A = A^{+\theta} \smallsetminus A^{-\theta}$ is the \define{$\theta$-boundary of $A$}, the set $\boundaryOf_\theta^- A = A \smallsetminus A^{-\theta}$ is the \define{internal $\theta$-boundary of $A$}, and the set $\boundaryOf_\theta^+ A = A^{+\theta} \smallsetminus A$ is the \define{external $\theta$-boundary of $A$}. (See \cite{wacker:growth:2017}.)

  A \define{semi-cellular automaton} is a quadruple $\mathcal{C} = \ntuple{\mathcal{R}, Q, N, \delta}$, where $\mathcal{R}$ is a cell space; $Q$, called \define{set of states}, is a set; $N$, called \define{neighbourhood}, is a subset of $G \modulo G_0$ such that $G_0 \cdot N \subseteq N$; and $\delta$, called \define{local transition function}, is a map from $Q^N$ to $Q$. A \define{local configuration} is a map $\ell \in Q^N$, a \define{global configuration} is a map $c \in Q^M$, and a \define{pattern} is a map $p \in Q^A$, where $A$ is a subset of $M$. The stabiliser $G_0$ acts on $Q^N$ on the left by $\bullet \from G_0 \times Q^N \to Q^N$, $(g_0, \ell) \mapsto [n \mapsto \ell(g_0^{-1} \cdot n)]$, and the group $G$ acts on the set of patterns on the left by 
  \begin{align*}
    \actsOnMap \from G \times \bigcup_{A \subseteq M} Q^A &\to     \bigcup_{A \subseteq M} Q^A,\\
                                                           (g, p) &\mapsto \left[
                                                                             \begin{aligned}
                                                                               g \actsOnPoint \domainOf(p) &\to     Q,\\
                                                                                                      m &\mapsto p(g^{-1} \actsOnPoint m).
                                                                             \end{aligned}
                                                                           \right]
  \end{align*}
%
  The \define{global transition function of $\mathcal{C}$} is the map $\Delta \from Q^M \to Q^M$, $c \mapsto [m \mapsto \delta(n \mapsto c(m \isSemiActedUponBy n))]$.

  A \define{cellular automaton} is a semi-cellular automaton $\mathcal{C} = \ntuple{\mathcal{R}, Q, N, \delta}$ such that $\delta$ is \define{$\bullet$-invariant}, which means that, for each $g_0 \in G_0$, we have $\delta(g_0 \bullet \blank) = \delta(\blank)$. Its global transition function is $\actsOnMap$-equivariant, which means that, for each $g \in G$, we have $\Delta(g \actsOnMap \blank) = g \actsOnMap \Delta(\blank)$.

  A subgroup $H$ of $G$ is \define{$\mathcal{K}$-big} if the set $\setOf{g_{m_0, m} \suchThat m \in M}$ is included in $H$. A \define{big-cellular automaton} is a semi-cellular automaton $\mathcal{C} = \ntuple{\mathcal{R}, Q, N, \delta}$ such that, for some $\mathcal{K}$-big subgroup $H$ of $G$, the local transition function $\delta$ is \define{$\bullet_{G_0 \cap H}$-invariant}, which means that, for each $h_0 \in G_0 \cap H$, we have $\delta(h_0 \bullet \blank) = \delta(\blank)$. Its global transition function is $\actsOnMap_H$-equivariant, which means that, for each $h \in H$, we have $\Delta(h \actsOnMap \blank) = h \actsOnMap \Delta(\blank)$. Note that each $\mathcal{K}$-big subgroup of $G$ includes the subgroup of $G$ generated by $\setOf{g_{m_0, m} \suchThat m \in M}$ and that hence a semi-cellular automaton is a big-cellular automaton if and only if its local transition function is $\bullet_{G_0 \cap \gen{g_{m_0, m} \suchThat m \in M}}$-invariant. (See \cite{wacker:automata:2016}.)


  For each $A \subseteq M$, let $\pi_A \from Q^M \to Q^A$, $c \mapsto c\restrictedTo_A$.

  \subsubsection*{Context.} In the present paper, let $\mathcal{R} = \ntuple{\mathcal{M}, \mathcal{K}} = \ntuple{\ntuple{M, G, \actsOnPoint}, \ntuple{m_0, \family{g_{m_0, m}}_{m \in M}}}$ be a finitely right generated cell space such that the stabiliser $G_0$ of $m_0$ under $\actsOnPoint$ is finite; let $S$ be a finite and symmetric right generating set of $\mathcal{R}$; let $H$ be a $\mathcal{K}$-big subgroup of $G$; let $H_0$ be the stabiliser of $m_0$ under $\actsOnPoint\restrictedTo_{H \times M}$, which is $H \cap G_0$; for each cell $m \in M$, let $H_{m_0, m}$ be the transporter of $m_0$ to $m$ under $\actsOnPoint\restrictedTo_{H \times M}$; let $Q$ be a finite set; let $Q^M$ be equipped with the prodiscrete topology; and identify $M$ with $G \modulo G_0$ by $\iota \from m \mapsto G_{m_0, m}$. Moreover, we omit the subscript $S$, in particular, instead of $d_S$ we write $d$, instead of $\absoluteValueOf{\blank}_S$ we write $\absoluteValueOf{\blank}$, instead of $\ball_S$ we write $\ball$, and instead of $\sphere_S$ we write $\sphere$. 


  \section{Full Shifts} 
  \label{sec:full-shifts}

  \subsubsection*{Contents.} The full shift is the set of global configurations, the points of the full shift (see \cref{def:full-shift}). A pattern is a map from a subset of cells to the set of states (see \cref{def:pattern}), its size is the cardinality of its domain (see \cref{def:pattern:size}), it is empty if its domain is empty (see \cref{def:pattern:empty}), it is finite and called \emph{block} if its domain is finite (see \cref{def:pattern:finite,def:block}), and restricting its domain yields a subpattern (see \cref{def:pattern:subpattern}). A subset of the full shift is shift-invariant if it is invariant under a group that contains the coordinates (see \cref{def:shift-invariant}). A pattern centred at the origin can be shifted to a new centre by sort of acting on the new centre (see \cref{def:induced-right-semi-action}).


  \begin{definition}
  \label{def:full-shift}
    The set $Q^M$ is called \define{full shift}\graffito{full shift $Q^M$} and each element $c \in Q^M$ is called \define{point}\graffito{point $c$}.
  \end{definition}

  \begin{example}[{\cite[Definition~1.1.1]{lind:marcus:1995}}]
  \label{ex:full-shift-over-integers}
    Let $\mathcal{M}$ be the left homogeneous space $\ntuple{\Z, \Z, +}$, let $\mathcal{K}$ be the coordinate system $\ntuple{0, \family{z}_{z \in \Z}}$, let $\mathcal{R}$ be the cell space $\ntuple{\mathcal{M}, \mathcal{K}}$, let $S$ be the set $\setOf{-1, 1}$, let $H$ be the only $\mathcal{K}$-big subgroup $\Z$ of $\Z$, and let $Q$ be the binary set $\setOf{0, 1}$. The stabiliser $\Z_0$ of $0$ under $+$ is the singleton set $\setOf{0}$; under the identification of $\Z$ with $\Z \modulo \Z_0$ by $z \mapsto z + \Z_0$, the right semi-action of $\Z \modulo \Z_0$ on $\Z$ is but $+$; and the set $S$ is a finite and symmetric right generating set of $\mathcal{R}$. The full shift $Q^\Z$ is the usual full $2$-shift considered in symbolic dynamics and its points are called \define{bi-infinite binary sequences}.
  \end{example}

  \begin{remark}
    There is a bijective map $\varphi$ from $Q$ to $\Z_{\absoluteValueOf{Q}}\ (= \setOf{0, 1, \dotsc, \absoluteValueOf{Q} - 1})$. It induces the bijective map
    \begin{align*}
      \Phi \from Q^M &\to     \Z_{\absoluteValueOf{Q}}^M,\\
                   c &\mapsto \big[m \mapsto \varphi\parens[\big]{c(m)}\big]. 
    \end{align*}
  \end{remark}

  \begin{definition} 
  \label{def:pattern}
    Let $A$ be a subset of $M$ and let $p$ be a map from $A$ to $Q$. The map $p$ is called \define{$A$-pattern}\graffito{$A$-pattern $p$} and the set $\domainOf(p) = A$ is called \define{domain of $p$}\graffito{domain $\domainOf(p)$ of $p$}. 
  \end{definition}

  \begin{definition}
  \label{def:pattern:size}
    Let $p$ be an $A$-pattern. The cardinal number $\absoluteValueOf{p} = \absoluteValueOf{A}$ is called \define{size of $p$}\graffito{size $\absoluteValueOf{p}$ of $p$}.
  \end{definition}


  \begin{definition}
  \label{def:pattern:empty}
    The $\emptyset$-pattern is called \define{empty}\graffito{empty pattern $\varepsilon$}. 
  \end{definition}

  \begin{remark}
    The empty pattern is the only one of size $0$.
  \end{remark}

  \begin{definition}
  \label{def:pattern:finite}
    Let $u$ be an $F$-pattern. It is called \define{finite}\graffito{finite pattern $u$} and \define{$F$-block}\graffito{$F$-block $u$} if and only if its domain $F$ is finite.
  \end{definition}

  \begin{definition}
  \label{def:block}
    The set $\setOf{u \in Q^F \suchThat F \subseteq M \text{ finite}}$ of all blocks is denoted by $Q^*$\graffito{set $Q^*$ of all blocks}.
  \end{definition}

  \begin{definition} 
  \label{def:pattern:subpattern}
    Let $p$ be an $A$- and let $p'$ be an $A'$-pattern. The pattern $p$ is called \define{subpattern of $p'$}\graffito{subpattern $p$ of $p'$} if and only if $A \subseteq A'$ and $p = p'\restrictedTo_A$.
  \end{definition}

  \begin{definition}
  \label{def:shift-invariant}
    Let $X$ be a subset of $Q^M$. It is called \define{shift-invariant}\graffito{shift-invariant} if and only if
    \begin{equation*}
      \ForEach h \in H \Holds h \actsOnMap X = X. 
    \end{equation*}
  \end{definition}

  \begin{remark}
    Shift-invariance is the same as $\actsOnMap_H$-invariance.
  \end{remark}

  \begin{remark}
  \label{rem:subseteq-sufficient-for-shift-invariance}
    The set $X$ is shift-invariant if and only if
    \begin{equation*}
      \ForEach h \in H \Holds h \actsOnMap X \subseteq X.
    \end{equation*}
  \end{remark}

  \begin{definition}
  \label{def:induced-right-semi-action}
    The map
    \begin{align*}
      \actsByItsCoordinateOn \from M \times \bigcup_{A \subseteq M} Q^A &\to     \bigcup_{A \subseteq M} Q^A,\\
                                                                  (m, p) &\mapsto \left[
                                                                                     \begin{aligned}
                                                                                       m \isSemiActedUponBy \domainOf(p) &\to     Q,\\
                                                                                                m \isSemiActedUponBy a &\mapsto p(a),
                                                                                     \end{aligned}
                                                                                   \right] 
    \end{align*}
    broadly speaking, maps a point $m$ and a pattern $p$ that is centred at $m_0$ to the corresponding pattern centred at $m$.
  \end{definition}

  \begin{remark}
    Let $p$ be an $A$-pattern and let $m$ be an element of $M$. Then, $m \actsByItsCoordinateOn p = g_{m_0, m} \actsOnMap p$ and $\domainOf(m \actsByItsCoordinateOn p) = g_{m_0, m} \actsOnPoint \domainOf(p)$. 
  \end{remark}

  \begin{remark}
  \label{rem:shift-invariance-induces-invariance-under-right-semi-action}
    Let $X$ be a shift-invariant subset of $Q^M$. Then,
    \begin{equation*}
      \ForEach m \in M \Holds m \actsByItsCoordinateOn X = X. 
    \end{equation*}
  \end{remark}

  \section{Shift Spaces} 
  \label{sec:shift-spaces}

  \subsubsection*{Contents.} A pattern semi-occurs in another pattern if a rotation of it occurs in the other pattern (see \cref{def:occurs-semi-occurs,rem:semi-occurs-at,rem:semi-occurs}). It is allowed in a subset of the full shift if it semi-occurs in one of its points and it is forbidden otherwise (see \cref{def:allowed-forbidden}). The set of points of the full shift in which each block of a given set is forbidden is generated by that set (see \cref{def:generated-by-forbidden-blocks}) and it is a shift space (see \cref{def:shift-space}), which is shift-invariant (see \cref{lem:shift-space-is-shift-invariant}), closed (see \cref{lem:if-restrictions-to-balls-are-in-shift-then-so-is-pattern}), and compact (see \cref{lem:shift-space-iff-invariant-and-compact}).

  If there is a finite generating set, then the shift space is of finite type (see \cref{def:of-finite-type}), the radius $\kappa$ of a ball that includes the domains of the patterns of the generating set is its memory and it itself is called $\kappa$-step (see \cref{def:step-and-memory,lem:of-finite-type-implies-step}), and its points are characterised by restrictions to balls with its memory as radius (see \cref{lem:characterisation-of-kappa-step-subshifts}). Finitely many points of shift spaces of finite type can be, in various ways, cut into pieces and glued together to construct new points, as long as the pieces agree on a big enough boundary of the cuts (see \cref{lem:overlapping-global-configurations-can-be-glued,cor:overlapping-global-configurations-can-be-glued,cor:overlapping-patterns-can-be-glued}). A shift space is strongly irreducible if allowed finite patterns that are at least a certain distance apart can be embedded in the same point (see \cref{def:kappa-strongly-irreducible,def:strongly-irreducible}). It has bounded propagation if finite patterns are allowed whenever their restrictions to balls of a certain radius are allowed (see \cref{def:rho-bounded-propagation,def:bounded-propagation}). Such spaces are strongly irreducible and of finite type (see \cref{lem:bounded-propagation-implies-strong-irreducibility-and-finite-type}).

  A map from a shift space to another one is local if it is uniformly and locally determined in each cell (see \cref{def:kappa-local-map,def:local-map}). Such maps are global transition functions of big-cellular automata with finite neighbourhoods (see \cref{rem:local-map-iff-cellular-automaton}), their domains and codomains can be simultaneously restricted to a subset of cells and its interior (see \cref{def:restriction-of-local-map}), and their images are shift spaces (see \cref{lem:image-of-local-map-is-subshift}). The difference of two points of the full shift is the set of cells in which they differ (see \cref{def:difference-of-two-points-of-the-full-shift}) and a local map is pre-injective if it is injective on points with finite support (see \cref{def:pre-injective-for-local-maps}). A local map that has a local inverse is a conjugacy (see \cref{def:conjugacy}), and its domain and codomain are conjugate (see \cref{def:conjugate}). Entropy is invariant under conjugacy (see \cref{lem:entropy-invariant-under-conjugacy}). A subshift has the Moore property if each surjective local map is pre-injective, and the Myhill property if the converse holds (see \cref{def:Moore-and-Myhill-properties}). Both these properties are invariant under conjugacy (see \cref{rem:Moore-and-Myhill-are-invariant-under-conjugacy}).

  \subsubsection*{Body.} A pattern occurs in another pattern if a translation of it coincides with a subpattern the other pattern and it semi-occurs in another pattern if a rotation and translation of it coincides with a subpattern of the other pattern, as defined in

  \begin{definition}
  \label{def:occurs-semi-occurs}
    Let $p$ be an $A$-pattern and let $p'$ be an $A'$-pattern.
    \begin{enumerate} 
      \item Let $m$ be an element of $M$. The pattern $p$ is said to \define{occur at $m$ in $p'$}\graffito{$p$ occurs at $m$ in $p'$} and we write $p \occursIn_m p'$\graffito{$p \occursIn_m p'$} if and only if
            \begin{equation*}
              m \isSemiActedUponBy A \subseteq A' \text{ and } m \actsByItsCoordinateOn p = p'\restrictedTo_{m \isSemiActedUponBy A}.
            \end{equation*}
            And it is said to \define{semi-occur at $m$ in $p'$}\graffito{$p$ semi-occurs at $m$ in $p'$} and we write $p \semiOccursIn_m p'$\graffito{$p \semiOccursIn_m p'$} if and only if
            \begin{equation*}
              \Exists h \in H_{m_0, m} \Holds h \actsOnPoint A \subseteq A' \text{ and } h \actsOnMap p = p'\restrictedTo_{h \actsOnPoint A}.
            \end{equation*}
      \item The pattern $p$ is said to \define{occur in $p'$}\graffito{$p$ occurs in $p'$} and we write $p \occursIn p'$\graffito{$p \occursIn p'$} if and only if 
            \begin{equation*}
              \Exists m \in M \SuchThat p \occursIn_m p'.
            \end{equation*}
            And it is said to \define{semi-occur in $p'$}\graffito{$p$ semi-occurs in $p'$} and we write $p \semiOccursIn p'$\graffito{$p \semiOccursIn p'$} if and only if 
            \begin{equation*}
              \Exists m \in M \SuchThat p \semiOccursIn_m p'. 
            \end{equation*}
    \end{enumerate}
  \end{definition}

  \begin{remark} 
  \label{rem:groups:occurs}
    Let $\mathcal{R}$ be the cell space $\ntuple{\ntuple{G, G, \cdot}, \ntuple{e_G, \family{g}_{g \in G}}}$, where $G$ is a group, $\cdot$ is its operation, and $e_G$ is its neutral element. Then, $G_0 = \setOf{e_G}$, $\isSemiActedUponBy = \cdot$, $\actsByItsCoordinateOn = \actsOnMap$, and, for each element $g \in G$, we have $H_{e_G, g} = \setOf{g}$. Hence, the notions \emph{occurs} and \emph{semi-occurs} are identical, and they are the common notion of \emph{occurrence} as used in \cite{fiorenzi:2003}.
  \end{remark}

  Semi-occurrence can be characterised in many ways, each illuminating a different aspect, some of which are given in

  \begin{remark}
  \label{rem:semi-occurs-at}
    Let $p$ be an $A$-pattern, let $p'$ be an $A'$-pattern, and let $m$ be an element of $M$. The following statements are equivalent:
    \begin{enumerate}
      \item $\displaystyle p \semiOccursIn_m p'$; 
      \item $\displaystyle \Exists h_0 \in H_0 \SuchThat h_0 \actsOnMap p \occursIn_m p'$;
      \item $\displaystyle \Exists h_0 \in H_0 \SuchThat g_{m_0, m} h_0 \actsOnPoint A \subseteq A' \text{ and } g_{m_0, m} h_0 \actsOnMap p = p'\restrictedTo_{g_{m_0, m} h_0 \actsOnPoint A}$;
      \item $\displaystyle \Exists h \in H_{m, m_0} \SuchThat A \subseteq h \actsOnPoint A' \text{ and } (h \actsOnMap p')\restrictedTo_A = p$;
      \item $\displaystyle \Exists h_0 \in H_0 \SuchThat A \subseteq h_0 g_{m_0, m}^{-1} \actsOnPoint A' \text{ and } (h_0 g_{m_0, m}^{-1} \actsOnMap p')\restrictedTo_A = p$. 
    \end{enumerate}
  \end{remark}

  \begin{remark}
  \label{rem:semi-occurs}
    Let $p$ be an $A$-pattern and let $p'$ be an $A'$-pattern. The following statements are equivalent:
    \begin{enumerate}
      \item $\displaystyle p \semiOccursIn p'$; 
      \item $\displaystyle \Exists h_0 \in H_0 \SuchThat h_0 \actsOnMap p \occursIn p'$;
      \item $\displaystyle \Exists h \in H \SuchThat h \actsOnPoint A \subseteq A' \text{ and } h \actsOnMap p = p'\restrictedTo_{h \actsOnPoint A}$;
      \item $\displaystyle \Exists h \in H \SuchThat A \subseteq h \actsOnPoint A' \text{ and } (h \actsOnMap p')\restrictedTo_A = p$. 
    \end{enumerate}
  \end{remark}

  %


  A pattern is allowed in a subset of the full shift if it semi-occurs in one of its points and forbidden otherwise, as defined in

  \begin{definition}
  \label{def:allowed-forbidden}
    Let $X$ be a subset of $Q^M$ and let $p$ be an $A$-pattern. The pattern $p$ is called
    \begin{enumerate}
      \item \define{allowed in $X$}\graffito{pattern $p$ allowed in $X$} if and only if 
            \begin{equation*}
              \Exists x \in X \SuchThat p \semiOccursIn x;
            \end{equation*}
      \item \define{forbidden in $X$}\graffito{pattern $p$ forbidden in $X$} if and only if 
            \begin{equation*}
              \ForEach x \in X \Holds p \not\semiOccursIn x. 
            \end{equation*}
    \end{enumerate}
  \end{definition}

  The greatest subset of the full shift with respect to inclusion in which each block of a given set of blocks is forbidden is said to be generated by the set of blocks, as defined in

  \begin{definition}
  \label{def:generated-by-forbidden-blocks}
    Let $\mathfrak{F}$ be a subset of $Q^*$. The set
    \begin{equation*}
      \gen{\mathfrak{F}} = \setOf{c \in Q^M \suchThat \ForEach \mathfrak{f} \in \mathfrak{F} \Holds \mathfrak{f} \not\semiOccursIn c} \mathnote{set $\gen{\mathfrak{F}}$ generated by $\mathfrak{F}$}
    \end{equation*}
    is said to be \define{generated by $\mathfrak{F}$}. 
  \end{definition}


  A shift space is a subset of the full shift that is generated by a set of blocks, as defined in

  \begin{definition} 
  \label{def:shift-space}
    Let $X$ be a subset of $Q^M$. It is called \define{shift space}\graffito{shift space $X$} and \define{subshift of $Q^M$}\graffito{subshift $X$ of $Q^M$} if and only if there is a subset $\mathfrak{F}$ of $Q^*$ such that $\gen{\mathfrak{F}} = X$. 
  \end{definition}

  \begin{example}[Full Shift] 
  \label{ex:shift:full}
    Because $\gen{\emptyset} = Q^M$, the set $Q^M$ is a shift space.
  \end{example}

  \begin{example}[Empty Shift] 
  \label{ex:shift:empty}
    If we identify each $q \in Q$ with the $\ball(m_0)$-block $[m_0 \mapsto q]$, then $\gen{Q} = \emptyset$ and hence the set $\emptyset$ is a shift space.
  \end{example}

  \begin{example}[{Golden Mean Shift \cite[Example~1.2.3]{lind:marcus:1995}}]
  \label{ex:shift:golden-mean}
    In the situation of \cref{ex:full-shift-over-integers}, the set $X$ of all bi-infinite binary sequences with no two $1$'s next to each other is the shift space known as \define{golden mean shift}. It is for example generated by the forbidden block $11$. 
  \end{example}

  \begin{example}[{Even Shift \cite[Example~1.2.4]{lind:marcus:1995}}]
  \label{ex:shift:even}
    In the situation of \cref{ex:full-shift-over-integers}, the set $X$ of all bi-infinite binary sequences such that, between any two occurrences of $1$'s, there are an even number of $0$'s, is the shift space known as \define{even shift}. It is for example generated by the forbidden blocks $1 0^{2k + 1} 1$, for $k \in \N_0$.
  \end{example}


  Each shift space is shift-invariant, which is shown in

  \begin{lemma}
  \label{lem:shift-space-is-shift-invariant}
    Let $X$ be a subshift of $Q^M$. It is shift-invariant.
  \end{lemma}

  \begin{proof} 
    There is a subset $\mathfrak{F}$ of $Q^*$ such that $\gen{\mathfrak{F}} = X$. Let $x \in X$ and let $h \in H$. Suppose that there is an $\mathfrak{f} \in \mathfrak{F}$ such that $\mathfrak{f} \semiOccursIn h \actsOnMap x$. Then, there is an $h' \in H$ such that $h' \actsOnMap \mathfrak{f} = (h \actsOnMap x)\restrictedTo_{h' \actsOnPoint \domainOf(\mathfrak{f})}$. Thus, because $(h \actsOnMap x)\restrictedTo_{h' \actsOnPoint \domainOf(\mathfrak{f})} = h \actsOnMap (x\restrictedTo_{h^{-1} h' \actsOnPoint \domainOf(\mathfrak{f})})$, we have $h^{-1} h' \actsOnMap \mathfrak{f} = x\restrictedTo_{h^{-1} h' \actsOnPoint \domainOf(\mathfrak{f})}$. Hence, because $h^{-1} h' \in H$, we have $\mathfrak{f} \semiOccursIn x$, which contradicts that $x \in \gen{\mathfrak{F}}$. Therefore, contrary to the supposition, for each $\mathfrak{f} \in \mathfrak{F}$, we have $\mathfrak{f} \not\semiOccursIn h \actsOnMap x$. Hence, $h \actsOnMap x \in X$. Therefore, $h \actsOnMap X \subseteq X$. In conclusion, according to \cref{rem:subseteq-sufficient-for-shift-invariance}, the subshift $X$ is shift-invariant.
  \end{proof}

  \begin{example}[{Shift-Invariant Non-Shift \cite[Example~1.2.10]{lind:marcus:1995}}]
  \label{ex:shift-invariant-non-shift}
    In the situation of \cref{ex:full-shift-over-integers}, the set $X$ of all bi-infinite binary sequences in which the symbol $1$ occurs exactly once is shift-invariant, but it is not a shift space.
  \end{example}

  Patterns with the same domain can all be restricted to some subdomain, as done in

  \begin{definition}
    Let $A$ be a subset of $M$, let $B$ be a subset of $A$, and let $P$ be a subset of $Q^A$. The set
    \begin{equation*}
      P_B = \setOf{p\restrictedTo_B \suchThat p \in P} \mathnote{$P_B$} 
    \end{equation*}
    is the set of all $B$-subpatterns of patterns of $P$.
  \end{definition}

  Because subshifts are shift-invariant, restrictions to patterns behave nicely with translations and rotations, as remarked in

  \begin{remark}
  \label{rem:pattern-belongs-to-shift-if-and-only-if-translated-pattern-belongs}
    Let $X$ be a subshift of $Q^M$, let $A$ be a subset of $M$, let $h$ be an element of $H$, and let $m$ be an element of $M$. For each $A$-pattern $p$, we have $p \in X_A$ if and only if $h \actsOnMap p \in X_{h \actsOnPoint A}$, and $p \in X_A$ if and only if $m \actsByItsCoordinateOn p \in X_{m \isSemiActedUponBy A}$. In particular, $h \actsOnMap X_A = X_{h \actsOnPoint A}$ and $m \actsByItsCoordinateOn X_A = X_{m \isSemiActedUponBy A}$. And, if $A$ is finite, then $\absoluteValueOf{X_A} = \absoluteValueOf{X_{h \actsOnPoint A}}$ and $\absoluteValueOf{X_A} = \absoluteValueOf{X_{m \isSemiActedUponBy A}}$. 
  \end{remark}

  Because shift spaces are generated by forbidden blocks, which have finite domains, a point of the full shift belongs to a shift space if and only if its subpatterns do, which is shown in 

  \begin{lemma}
  \label{lem:if-restrictions-to-balls-are-in-shift-then-so-is-pattern}
    Let $X$ be a subshift of $Q^M$ and let $c$ be a point of $Q^M$. Then, $c \in X$ if and only if 
    \begin{equation}
    \label{eq:if-restrictions-to-balls-are-in-shift-then-so-is-pattern}
      \ForEach \rho \in \N_0 \Holds c\restrictedTo_{\ball(\rho)} \in X_{\ball(\rho)}. 
    \end{equation}
  \end{lemma}

  \begin{proof}
    If $c \in X$, then \cref{eq:if-restrictions-to-balls-are-in-shift-then-so-is-pattern} holds. From now on, let \cref{eq:if-restrictions-to-balls-are-in-shift-then-so-is-pattern} hold. Because $X$ is a subshift, there is a subset $\mathfrak{F}$ of $Q^*$ such that $\gen{\mathfrak{F}} = X$. Let $\mathfrak{f} \in \mathfrak{F}$. Suppose that $\mathfrak{f}$ semi-occurs in $c$. Then, because $\absoluteValueOf{\mathfrak{f}} < \infty$, according to \cite[Remark~5.6]{wacker:growth:2017}, there is a $\rho \in \N_0$ such that $\mathfrak{f}$ semi-occurs in $c\restrictedTo_{\ball(\rho)}$. Hence, $c\restrictedTo_{\ball(\rho)} \notin X_{\ball(\rho)}$, which contradicts \cref{eq:if-restrictions-to-balls-are-in-shift-then-so-is-pattern}. Therefore, $\mathfrak{f}$ does not semi-occur in $c$. In conclusion, $c \in X$.
  \end{proof}

  Shift spaces are characterised by shift-invariance and compactness, which is shown in

  \begin{lemma} 
  \label{lem:shift-space-iff-invariant-and-compact}
    Let $X$ be a subset of $Q^M$. It is a shift space if and only if it is shift-invariant and compact.
  \end{lemma}

  \begin{proof}
    First, let $X$ be a shift space. Then, according to \cref{lem:shift-space-is-shift-invariant}, it is shift-invariant. Moreover, let $\sequence{x_k}_{k \in \N_+}$ be a sequence in $X$ that converges to a point $c \in Q^M$. Then, for each $\rho \in \N_0$, there is a $k \in \N_+$ such that $c\restrictedTo_{\ball(\rho)} = x_k\restrictedTo_{\ball(\rho)} \in X_{\ball(\rho)}$. 
    Thus, according to \cref{lem:if-restrictions-to-balls-are-in-shift-then-so-is-pattern}, we have $c \in X$. Hence, $X$ is closed. And, according to \cite[the first paragraph in Section~1.8]{ceccherini-silberstein:coornaert:2010}, the set $Q^M$ is compact. Therefore, $X$ is compact.

    Secondly, let $X$ be shift-invariant and compact. Then, $X$ is closed and $Q^M \smallsetminus X$ is open. Hence, for each $c \in Q^M \smallsetminus X$, there is a $\rho_c \in \N_0$ such that $\cylinder(c\restrictedTo_{\ball(\rho_c)}) \subseteq Q^M \smallsetminus X$. Put $\mathfrak{F} = \setOf{c\restrictedTo_{\ball(\rho_c)} \suchThat c \in Q^M \smallsetminus X}$.

    Let $x \in X$. Suppose that there is an $\mathfrak{f} \in \mathfrak{F}$ such that $\mathfrak{f} \semiOccursIn x$. Then, according to \cref{rem:semi-occurs}, there is an $h \in H$ such that $(h \actsOnMap x)\restrictedTo_{\domainOf(\mathfrak{f})} = \mathfrak{f}$. Thus, $h \actsOnMap x \in \cylinder(\mathfrak{f}) \subseteq Q^M \smallsetminus X$. And, because $X$ is shift-invariant, we also have $h \actsOnMap x \in X$, which contradicts that $h \actsOnMap x \in Q^M \smallsetminus X$. Hence, contrary to the supposition, for each $\mathfrak{f} \in \mathfrak{F}$ we have $\mathfrak{f} \not\semiOccursIn x$. Therefore, $X \subseteq \gen{\mathfrak{F}}$.

    Let $c \in \gen{\mathfrak{F}}$. Suppose that $c \notin X$. Then, $c\restrictedTo_{\ball(\rho_c)} \in \mathfrak{F}$. Thus $c \notin \gen{\mathfrak{F}}$, which contradicts that $c \in \gen{\mathfrak{F}}$. Hence, $c \in X$. Therefore, $\gen{\mathfrak{F}} \subseteq X$.

    Altogether, $\gen{\mathfrak{F}} = X$. In conclusion, $X$ is a shift space.
  \end{proof}

  \begin{remark}
    Compactness cannot be omitted in the equivalence. For example, the subset of $\setOf{0, 1}^\Z$ from \cref{ex:shift-invariant-non-shift} is shift-invariant but not a shift space.
  \end{remark}

  A shift space of finite type is one that is generated by finitely many forbidden blocks, as defined in

  \begin{definition} 
  \label{def:of-finite-type}
    Let $X$ be a subshift of $Q^M$. It is said to be \define{of finite type}\graffito{of finite type} if and only if there is a finite subset $\mathfrak{F}$ of $Q^*$ such that $\gen{\mathfrak{F}} = X$.
  \end{definition}

  \begin{example}[Of Finite Type or Not]
  \label{ex:of-finite-type-or-not}
    As is apparent from their definition, the full shift (\cref{ex:shift:full}), the empty shift (\cref{ex:shift:empty}), and the golden mean shift (\cref{ex:shift:golden-mean}) are of finite type. According to \cite[Example~2.1.5]{lind:marcus:1995}, the even shift (\cref{ex:shift:even}) is \emph{not} of finite type.
  \end{example}

  A $\kappa$-step shift space is one that is generated by forbidden $\ball(\kappa)$-blocks, as defined in

  \begin{definition} 
  \label{def:step-and-memory}
    Let $X$ be a subshift of $Q^M$ and let $\kappa$ be a non-negative integer. The subshift $X$ is called \define{$\kappa$-step}\graffito{$\kappa$-step} and the integer $\kappa$ is called \define{memory of $X$}\graffito{memory $\kappa$ of $X$} if and only if there is a subset $\mathfrak{F}$ of $Q^{\ball(\kappa)}$ such that $\gen{\mathfrak{F}} = X$.
  \end{definition}

  \begin{remark}
  \label{rem:k-step-for-greater-k}
    Let $X$ be a $\kappa$-step subshift of $Q^M$. Because the set $Q^{\ball(\kappa)}$ is finite, the subshift $X$ is of finite type. And, for each non-negative integer $\kappa'$ such that $\kappa' \geq \kappa$, the subshift $X$ is $\kappa'$-step. 
  \end{remark}

  A shift space of finite type is $\kappa$-step, where $\kappa$ is the radius of a ball that includes all domains of a finite generating set of the shift space, which is shown in

  \begin{lemma}
  \label{lem:of-finite-type-implies-step}
    Let $X$ be a subshift of $Q^M$ of finite type. There is a non-negative integer $\kappa$ such that $X$ is $\kappa$-step.
  \end{lemma}

  \begin{proof}
    Because $X$ is of finite type, there is a finite subset $\mathfrak{F}$ of $Q^*$ such that $\gen{\mathfrak{F}} = X$. And, because the set $\mathfrak{F}$ is finite, according to \cite[Remark~5.6]{wacker:growth:2017}, there is a non-negative integer $\kappa$ such that, for each $\mathfrak{f} \in \mathfrak{F}$, we have $\domainOf(\mathfrak{f}) \subseteq \ball(\kappa)$. Let $\mathfrak{F}'$ be the set $\setOf{p \in Q^{\ball(\kappa)} \suchThat \Exists \mathfrak{f} \in \mathfrak{F} \SuchThat p\restrictedTo_{\domainOf(\mathfrak{f})} = \mathfrak{f}}$. Then, $\gen{\mathfrak{F}'} = \gen{\mathfrak{F}} = X$. In conclusion, $X$ is $\kappa$-step. 
  \end{proof}

  A shift space is $\kappa$-step
  if and only if it contains each point of the full shift whose restrictions to the balls of radius $\kappa$ are allowed patterns, which is shown in

  \begin{lemma} 
  \label{lem:characterisation-of-kappa-step-subshifts}
    Let $X$ be a subshift of $Q^M$ and let $\kappa$ be a non-negative integer. The subshift $X$ is $\kappa$-step if and only if 
    \begin{equation}
    \label{eq:characterisation-of-kappa-step-subshifts:char}
      \ForEach c \in Q^M \Holds \parens[\Big]{\parens[\big]{\ForEach m \in M \Holds c\restrictedTo_{\ball(m, \kappa)} \in X_{\ball(m, \kappa)}} \implies c \in X}. 
    \end{equation}
  \end{lemma}

  \begin{proof}
    First, let $X$ be $\kappa$-step. Then, there is an $\mathfrak{F} \subseteq Q^{\ball(\kappa)}$ such that $\gen{\mathfrak{F}} = X$. Furthermore, let $c \in Q^M$ such that
    \begin{equation}
    \label{eq:characterisation-of-kappa-step-subshifts:antedecent}
      \ForEach m \in M \Holds c\restrictedTo_{\ball(m, \kappa)} \in X_{\ball(m, \kappa)}.
    \end{equation}
    Suppose that there is an $\mathfrak{f} \in \mathfrak{F}$ such that $\mathfrak{f} \semiOccursIn c$. Then, there is an $m \in M$ such that $\mathfrak{f} \semiOccursIn_m c$. Thus, there is an $h \in H_{m_0, m}$ such that $h \actsOnMap \mathfrak{f} = c\restrictedTo_{h \actsOnPoint \ball(\kappa)}$. Moreover, because $h \actsOnPoint m_0 = m$, according to \cite[Lemma~5.9]{wacker:growth:2017}, we have $h \actsOnPoint \ball(\kappa) = \ball(m, \kappa)$. Hence, according to \cref{eq:characterisation-of-kappa-step-subshifts:antedecent}, we have $h \actsOnMap \mathfrak{f} = c\restrictedTo_{\ball(m, \kappa)} \in X_{\ball(m, \kappa)} = X_{h \actsOnPoint \ball(\kappa)}$. Therefore, there is an $x \in X$ such that $\mathfrak{f} \semiOccursIn x$, which contradicts that $\gen{\mathfrak{F}} = X$. In conclusion, for each $\mathfrak{f} \in \mathfrak{F}$, we have $\mathfrak{f} \not\semiOccursIn c$, and hence $c \in X$.

    Secondly, let \cref{eq:characterisation-of-kappa-step-subshifts:char} hold. Furthermore, let $\mathfrak{F} = Q^{\ball(\kappa)} \smallsetminus X_{\ball(\kappa)}$. We show below that $X \subseteq \gen{\mathfrak{F}}$ and $\gen{\mathfrak{F}} \subseteq X$. Hence, $\gen{\mathfrak{F}} = X$. In conclusion, $X$ is $\kappa$-step.

    \proofpart{Subproof of: $X \subseteq \gen{\mathfrak{F}}$}
    Let $c \in Q^M \smallsetminus \gen{\mathfrak{F}}$. Then, there is an $\mathfrak{f} \in \mathfrak{F}$ such that $\mathfrak{f} \semiOccursIn c$. Thus, there is an $m \in M$ and there is an $h \in H_{m_0, m}$ such that $h \actsOnMap \mathfrak{f} = c\restrictedTo_{\ball(m, \kappa)}$. Therefore, because $h \actsOnMap \blank$ is bijective, according to \cref{rem:pattern-belongs-to-shift-if-and-only-if-translated-pattern-belongs} and \cite[Lemma~5.9]{wacker:growth:2017}, we have $c\restrictedTo_{\ball(m, \kappa)} \in h \actsOnMap \mathfrak{F} = h \actsOnMap (Q^{\ball(\kappa)} \smallsetminus X_{\ball(\kappa)}) = (h \actsOnMap Q^{\ball(\kappa)}) \smallsetminus (h \actsOnMap X_{\ball(\kappa)}) = Q^{\ball(m, \kappa)} \smallsetminus X_{\ball(m, \kappa)}$. Hence, $c \in Q^M \smallsetminus X$. In conclusion, $X \subseteq \gen{\mathfrak{F}}$.

    \proofpart{Subproof of: $\gen{\mathfrak{F}} \subseteq X$}
    Let $c \in Q^M \smallsetminus X$. Then, according to \cref{eq:characterisation-of-kappa-step-subshifts:char}, there is an $m \in M$ such that $c\restrictedTo_{\ball(m, \kappa)} \in Q^{\ball(m, \kappa)} \smallsetminus X_{\ball(m, \kappa)}$. Furthermore, let $h = g_{m_0, m}$. Then, similar as above, $h^{-1} \actsOnMap c\restrictedTo_{\ball(m, \kappa)} \in Q^{\ball(\kappa)} \smallsetminus X_{\ball(\kappa)} = \mathfrak{F}$. Thus, there is an $\mathfrak{f} \in \mathfrak{F}$ such that $h \actsOnMap \mathfrak{f} = c\restrictedTo_{\ball(m, \kappa)}$. Hence, because $h \in H$, we have $\mathfrak{f} \semiOccursIn c$. Therefore, $c \in Q^M \smallsetminus \gen{\mathfrak{F}}$. In conclusion, $\gen{\mathfrak{F}} \subseteq X$.
  \end{proof} 

  If we cut holes in a point of a shift space of finite type that are far enough apart and fill these holes with pieces from other points of the shift space that agree on big enough boundaries of the holes with the holey point, then we still have a point of the shift space, which is shown in

  \begin{lemma} 
  \label{lem:overlapping-global-configurations-can-be-glued}
    Let $X$ be a $\kappa$-step subshift of $Q^M$, let $x$ be a point of $X$, let $\family{A_i}_{i \in I}$ be a family of subsets of $M$ such that the family $\family{A_i^{+2\kappa}}_{i \in I}$ is pairwise disjoint, and let $\family{x_i}_{i \in I}$ be a family of points of $X$ such that, for each index $i \in I$, we have $x_i\restrictedTo_{\boundaryOf_{2\kappa}^+ A_i} = x\restrictedTo_{\boundaryOf_{2\kappa}^+ A_i}$. The map $x\restrictedTo_{M \smallsetminus (\bigcup_{i \in I} A_i)} \times \coprod_{i \in I} x_i\restrictedTo_{A_i}$ is identical to $x\restrictedTo_{M \smallsetminus (\bigcup_{i \in I} A_i^{+2\kappa})} \times \coprod_{i \in I} x_i\restrictedTo_{A_i^{+2\kappa}}$ and a point of $X$.
  \end{lemma} 

  \begin{proof}
    Because $\family{A_i^{+2\kappa}}_{i \in I}$ is pairwise disjoint, so is $\family{A_i}_{i \in I}$. Hence, $x' = x\restrictedTo_{M \smallsetminus (\bigcup_{i \in I} A_i)} \times \coprod_{i \in I} x_i\restrictedTo_{A_i}$ is well-defined. Furthermore, for each $i \in I$, we have $x_i\restrictedTo_{A_i^{+2\kappa} \smallsetminus A_i} = x\restrictedTo_{A_i^{+2\kappa} \smallsetminus A_i}$. Hence, $x' = x\restrictedTo_{M \smallsetminus (\bigcup_{i \in I} A_i^{+2\kappa})} \times \coprod_{i \in I} x_i\restrictedTo_{A_i^{+2\kappa}}$. Moreover, because $X$ is $\kappa$-step, there is an $\mathfrak{F} \subseteq Q^{\ball(\kappa)}$ such that $\gen{\mathfrak{F}} = X$.

    Let $u \in Q^{\ball(\kappa)}$ semi-occur in $x'$. Then, there is an $m \in M$ and there is an $h \in H_{m_0, m}$ such that $h \actsOnMap u = x'\restrictedTo_{h \actsOnPoint \ball(\kappa)}$. Moreover, according to \cite[Lemma~5.9]{wacker:growth:2017} and \cite[Item~2 of Lemma~6.2]{wacker:growth:2017}, we have $h \actsOnPoint \ball(\kappa) = \ball(m, \kappa) = \setOf{m}^{+\kappa}$.
    \begin{description} 
      \item[Case 1] $\Exists i \in I \SuchThat m \in A_i^{+\kappa}$. Then, according to \cite[Item~3 of Lemma~6.4]{wacker:growth:2017}, we have $\setOf{m}^{+\kappa} \subseteq (A_i^{+\kappa})^{+\kappa} \subseteq A_i^{+2\kappa}$. Hence, $h \actsOnPoint \ball(\kappa) \subseteq A_i^{+2\kappa}$. Thus, $x'\restrictedTo_{h \actsOnPoint \ball(\kappa)} = x_i\restrictedTo_{h \actsOnPoint \ball(\kappa)}$. Therefore, $u$ semi-occurs in $x_i$. Hence, $u \notin \mathfrak{F}$.
      \item[Case 2] $m \in M \smallsetminus (\bigcup_{i \in I} A_i^{+\kappa})$. Then, according to \cite[Item~3 of Lemma~1]{wacker:garden:2016} 
      and \cite[Item~5 of Lemma~6.4]{wacker:growth:2017}, we have $\setOf{m}^{+\kappa} \subseteq (M \smallsetminus (\bigcup_{i \in I} A_i^{+\kappa}))^{+\kappa} \subseteq M \smallsetminus (\bigcup_{i \in I} (A_i^{+\kappa})^{-\kappa}) \subseteq M \smallsetminus (\bigcup_{i \in I} A_i^{+(\kappa - \kappa)}) = M \smallsetminus (\bigcup_{i \in I} A_i)$. Hence, $h \actsOnPoint \ball(\kappa) \subseteq M \smallsetminus (\bigcup_{i \in I} A_i)$. Thus, $x'\restrictedTo_{h \actsOnPoint \ball(\kappa)} = x\restrictedTo_{h \actsOnPoint \ball(\kappa)}$. Therefore, $u$ semi-occurs in $x$. Hence, $u \notin \mathfrak{F}$.
    \end{description}
    In either case, $u \notin \mathfrak{F}$. In conclusion, $x' \in X$.
  \end{proof}

  If we sew one part and respectively the other part of two sufficiently overlapping points of a shift space of finite type together, then we get a point of the shift space, which is shown in

  \begin{corollary} 
  \label{cor:overlapping-global-configurations-can-be-glued}
    Let $X$ be a $\kappa$-step subshift of $Q^M$, let $A$ be a subset of $M$, and let $x$ and $x'$ be two points of $X$ such that $x\restrictedTo_{\boundaryOf_{2\kappa}^+ A} = x'\restrictedTo_{\boundaryOf_{2\kappa}^+ A}$. The map $x\restrictedTo_A \times x'\restrictedTo_{M \smallsetminus A}$ is identical to $x\restrictedTo_{A^{+2\kappa}} \times x'\restrictedTo_{M \smallsetminus A^{+2\kappa}}$ and a point of $X$.
  \end{corollary}

  \begin{proof}
    This is a direct consequence of \cref{lem:overlapping-global-configurations-can-be-glued} with $\family{x_i}_{i \in I} = \family{x'}$.
  \end{proof}

  If two patterns with the same domain, that are allowed in a shift space of finite type, agree on a big enough boundary, then they can be identically extended to points of the shift space, which is shown in

  \begin{corollary} 
  \label{cor:overlapping-patterns-can-be-glued}
    Let $X$ be a $\kappa$-step subshift of $Q^M$, let $A$ be a subset of $M$, let $p$ and $p'$ be two patterns of $X_{A^{+2\kappa}}$ such that $p\restrictedTo_{\boundaryOf_{2\kappa}^+ A} = p'\restrictedTo_{\boundaryOf_{2\kappa}^+ A}$. There are two points $x$ and $x'$ of $X$ such that $x\restrictedTo_{A^{+2\kappa}} = p$, $x'\restrictedTo_{A^{+2\kappa}} = p'$, and $x\restrictedTo_{M \smallsetminus A} = x'\restrictedTo_{M \smallsetminus A}$.
  \end{corollary}

  \begin{proof}
    Because $p$, $p' \in X_{A^{+2\kappa}}$, there are $x''$, $x' \in X$ such that $x''\restrictedTo_{A^{+2\kappa}} = p$ and $x'\restrictedTo_{A^{+2\kappa}} = p'$. Because $p\restrictedTo_{\boundaryOf_{2\kappa}^+ A} = p'\restrictedTo_{\boundaryOf_{2\kappa}^+ A}$, we have $x''\restrictedTo_{\boundaryOf_{2\kappa}^+ A} = x'\restrictedTo_{\boundaryOf_{2\kappa}^+ A}$. Hence, according to \cref{lem:overlapping-global-configurations-can-be-glued}, we have $x = x''\restrictedTo_A \times x'\restrictedTo_{M \smallsetminus A} \in X$. Moreover, $x\restrictedTo_{M \smallsetminus A} = x'\restrictedTo_{M \smallsetminus A}$. And, because $x = x''\restrictedTo_{A^{+2\kappa}} \times x'\restrictedTo_{M \smallsetminus A^{+2\kappa}}$, we have $x\restrictedTo_{A^{+2\kappa}} = p$.
  \end{proof}



  A shift space is strongly irreducible if two allowed finite patterns that are far enough apart are embedded in a point of the shift space, which is defined in

  \begin{definition} 
  \label{def:kappa-strongly-irreducible}
    Let $X$ be a subshift of $Q^M$ and let $\kappa$ be a non-negative integer. The subshift $X$ is called \define{$\kappa$-strongly irreducible}\graffito{$\kappa$-strongly irreducible} if and only if for each tuple $(p, p')$ of finite patterns allowed in $X$ such that $d(\domainOf(p), \domainOf(p')) \geq \kappa + 1$, there is a point $x \in X$ such that $x\restrictedTo_{\domainOf(p)} = p$ and $x\restrictedTo_{\domainOf(p')} = p'$.
  \end{definition}

  \begin{remark}
  \label{rem:k-strongly-irreducible-for-greater-k}
    Let $X$ be a $\kappa$-strongly irreducible subshift of $Q^M$. For each non-negative integer $\kappa'$ such that $\kappa' \geq \kappa$, the subshift $X$ is $\kappa'$-strongly irreducible. 
  \end{remark}

  \begin{definition}
  \label{def:strongly-irreducible}
    Let $X$ be a subshift of $Q^M$. It is called \define{strongly irreducible}\graffito{strongly irreducible} if and only if there is a non-negative integer $\kappa$ such that it is $\kappa$-strongly irreducible.
  \end{definition}

  \begin{example}[Strongly Irreducible]
  \label{ex:strongly-irreducible} 
    The full shift (\cref{ex:shift:full}) and the empty shift (\cref{ex:shift:empty}) are $0$-strongly irreducible. The golden mean shift is $1$-strongly irreducible. The even shift is $2$-strongly irreducible. The generalised golden mean shifts (\cref{ex:generalised-golden-mean-shift}) and the shift space of \cref{ex:strongly-irreducible-of-finite-type-without-bounded-propagation} are strongly irreducible. Of these examples, according to \cref{ex:of-finite-type-or-not}, all but the even shift are of finite type.
  \end{example}

  \begin{example}[{Not Strongly Irreducible \cite[Example~4.6]{ceccherini-silberstein:coornaert:2012}}] 
  \label{ex:not-strongly-irreducible}
    In the situation of \cref{ex:full-shift-over-integers}, the set $X$ of all bi-infinite binary sequences with no two $0$'s and no two $1$'s next to each other is a shift space. It is for example generated by the forbidden blocks $00$ and $11$, in particular, it is of finite type. It consists of the two bi-infinite binary sequences with alternating $0$'s and $1$'s. And, it is not strongly irreducible; indeed, for each even and non-negative integer $\kappa$, the finite patterns $01$ and $10$ are allowed in $X$, the allowed patterns of size $\kappa$ are of the form $(01)^{\kappa/2}$ or $(10)^{\kappa/2}$, but the patterns $01(01)^{\kappa/2}10$ and $01(10)^{\kappa/2}10$ are not allowed and hence not embedded in a point of $X$.
  \end{example}


  A shift space has bounded propagation if a finite pattern is allowed whenever all restrictions of it to balls of a fixed radius are allowed, as defined in

  \begin{definition} 
  \label{def:rho-bounded-propagation}
    Let $X$ be a subshift of $Q^M$ and let $\rho$ be a non-negative integer. The subshift $X$ is said to have \define{$\rho$-bounded propagation} if and only if
    \begin{multline*}
      \ForEach F \subseteq M \text{ finite} \ForEach p \in Q^F \Holds\\
          \parens[\Big]{ \parens[\big]{\ForEach f \in F \Holds p\restrictedTo_{\ball(f, \rho) \cap F} \in X_{\ball(f, \rho) \cap F}} \implies p \in X_F }.
    \end{multline*}
  \end{definition}

  \begin{remark}
    Let $X$ be a subshift of $Q^M$ with $\rho$-bounded propagation. For each non-negative integer $\rho'$ such that $\rho' \geq \rho$, the subshift $X$ has $\rho'$-bounded propagation. 
  \end{remark}

  \begin{definition}
  \label{def:bounded-propagation}
    Let $X$ be a subshift of $Q^M$. It is said to have \define{bounded propagation} if and only if there is a non-negative integer $\rho$ such that it has $\rho$-bounded propagation.
  \end{definition}

  A subshift with bounded propagation is strongly irreducible and of finite type, which is shown in

  \begin{lemma} 
  \label{lem:bounded-propagation-implies-strong-irreducibility-and-finite-type}
    Let $X$ be a subshift of $Q^M$ with $\rho$-bounded propagation. It is $\rho$-strongly irreducible and $\rho$-step.
  \end{lemma}

  \begin{proof}
    Let $(p, p')$ be a tuple of finite patterns allowed in $X$ such that $d(\domainOf(p), \domainOf(p')) \geq \rho + 1$. Moreover, let $F = \domainOf(p) \cup \domainOf(p')$ and let $p'' = p \times p' \in Q^F$. Furthermore, let $f \in F$. Then, if $f \in \domainOf(p)$, then $\ball(f, \rho) \cap F \subseteq \domainOf(p)$; and, if $f \in \domainOf(p')$, then $\ball(f, \rho) \cap F \subseteq \domainOf(p')$. Thus, in either case, $p''\restrictedTo_{\ball(f, \rho) \cap F} \in X_{\ball(f, \rho) \cap F}$. Hence, because $X$ has $\rho$-bounded propagation, we have $p'' \in X_F$. Therefore, there is an $x \in X$ such that $x\restrictedTo_F = p''$, in particular, $x\restrictedTo_{\domainOf(p)} = p$ and $x\restrictedTo_{\domainOf(p')} = p'$. In conclusion, $X$ is $\rho$-strongly irreducible.

    Let $c \in Q^M$ such that, for each $m \in M$, we have $c\restrictedTo_{\ball(m, \rho)} \in X_{\ball(m, \rho)}$. Moreover, let $\rho' \in \N_0$. Then, for each $f \in \ball(\rho')$, we have $(c\restrictedTo_{\ball(\rho')})\restrictedTo_{\ball(f, \rho) \cap \ball(\rho')} = (c\restrictedTo_{\ball(f, \rho)})\restrictedTo_{\ball(\rho') \cap \ball(f, \rho)} \in (X_{\ball(f, \rho)})_{\ball(\rho') \cap \ball(f, \rho)} = X_{\ball(f, \rho) \cap \ball(\rho')}$. Thus, because $X$ has $\rho$-bounded propagation, we have $c\restrictedTo_{\ball(\rho')} \in X_{\ball(\rho')}$. Hence, according to \cref{lem:if-restrictions-to-balls-are-in-shift-then-so-is-pattern}, we have $c \in X$. In conclusion, according to \cref{lem:characterisation-of-kappa-step-subshifts}, the subshift $X$ is $\rho$-step.
  \end{proof}

  \begin{example}[Has Bounded Propagation or Not]
  \label{ex:bounded-propagation}
    The full shift (\cref{ex:shift:full}) and the empty shift (\cref{ex:shift:empty}) have $0$-bounded propagation. The golden mean shift (\cref{ex:shift:golden-mean}) has $1$-bounded propagation. The even shift (\cref{ex:shift:even}) is, according to \cref{ex:of-finite-type-or-not}, not of finite type and hence it is, according to \cref{lem:bounded-propagation-implies-strong-irreducibility-and-finite-type}, does not have bounded propagation. The shift of \cref{ex:not-strongly-irreducible} is not strongly irreducible and hence it is, according to \cref{lem:bounded-propagation-implies-strong-irreducibility-and-finite-type}, does not have bounded propagation.
  \end{example}

  \begin{example}[{Generalised Golden Mean Shifts \cite[Example~2.8]{ceccherini-silberstein:coornaert:2012}}] 
  \label{ex:generalised-golden-mean-shift}
    Let $q$ be a positive integer, let $Q$ be the set $\setOf{0, 1, \dotsc, q}$, let $k$ be a positive integer, let $\family{F_i}_{i \in \setOf{1, 2, \dotsc, k}}$ be a family of finite subsets of $M$ that contain $m_0$, let $\rho$ be the least non-negative integer such that $\bigcup_{i \in \setOf{1, 2, \dotsc, k}} F_i \subseteq \ball(\rho)$, and let $X$ be the $\rho$-step subshift $\gen{p \in Q^{\ball(\rho)} \suchThat \ForEach i \in \setOf{1, 2, \dotsc, k} \ForEach m \in F_i \Holds p(m) \neq 0}$ of $Q^M$. The subshift $X$ is called \define{generalised golden mean shift}, it is equal to $\setOf{c \in Q^M \suchThat \ForEach i \in \setOf{1, 2, \dotsc, k} \ForEach h \in H \Exists m \in h \actsOnPoint F_i \SuchThat c(m) = 0}$, and it has $\rho$-bounded propagation, in particular, according to \cref{lem:bounded-propagation-implies-strong-irreducibility-and-finite-type}, it is $\rho$-strongly irreducible.

    In the case that $\mathcal{R}$ is the cell space from \cref{ex:full-shift-over-integers}, $q = 1$, $k = 1$, and $F_1 = \setOf{0, 1}$, the non-negative integer $\rho$ is equal to $1$ and the generalised golden mean shift $X$ is equal to the golden mean shift from \cref{ex:shift:golden-mean}.
  \end{example}

  \begin{proof}[Proof of Bounded Propagation]
    Let $F$ be a finite subset of $M$, let $p$ be a pattern of $Q^F$ such that
    \begin{equation*}
      \ForEach f \in F \Holds p\restrictedTo_{\ball(f, \rho) \cap F} \in X_{\ball(f, \rho) \cap F},
    \end{equation*}
    and let $c$ be the point of $Q^M$ that is equal to $p$ on $F$ and identically $0$ on $M \smallsetminus F$. Moreover, let $i \in I$ and let $h \in H$. If $h \actsOnPoint F_i \nsubseteq F$, then, there is an $m \in (h \actsOnPoint F_i) \smallsetminus F$, and, by definition of $c$, we have $c(m) = 0$. Otherwise, if $h \actsOnPoint F_i \subseteq F$, then, because $m_0 \in F_i$, we have $h \actsOnPoint m_0 \in F$; thus, because $p\restrictedTo_{\ball(h \actsOnPoint m_0, \rho) \cap F} \in X_{\ball(h \actsOnPoint m_0, \rho) \cap F}$, there is a point $x \in X$ such that $p\restrictedTo_{\ball(h \actsOnPoint m_0, \rho) \cap F} = x\restrictedTo_{\ball(h \actsOnPoint m_0, \rho) \cap F}$; hence, by the characterisation of $X$, there is a cell $m \in h \actsOnPoint F_i$ such that $x(m) = 0$; and therefore, because $h \actsOnPoint F_i \subseteq (h \actsOnPoint \ball(\rho)) \cap F = \ball(h \actsOnPoint m_0, \rho) \cap F$, we have $c(m) = p(m) = x(m) = 0$. Hence, in either case, there is an $m \in h \actsOnPoint F_i$ such that $c(m) = 0$. Therefore, by the characterisation of $X$, we have $c \in X$ and hence $p = c\restrictedTo_F \in X_F$. In conclusion, $X$ has $\rho$-bounded propagation.
  \end{proof}

  \begin{example}[{\cite[Section 4, at the very end]{fiorenzi:2003}}]
  \label{ex:strongly-irreducible-of-finite-type-without-bounded-propagation}
    In the situation of \cref{ex:full-shift-over-integers}, the subshift $\gen{010, 111}$ of $\setOf{0, 1}^\Z$ is strongly irreducible and of finite type but does not have bounded propagation. 
  \end{example}

  A map from a shift space to another shift space is local if the image of a point is uniformly and locally determined in each cell, as defined in

  \begin{definition} 
  \label{def:kappa-local-map} 
    Let $X$ and $Y$ be two subshifts of $Q^M$, let $\Delta$ be a map from $X$ to $Y$, let $\kappa$ be a non-negative integer, let $N$ be a subset of $\ball(\kappa)$ such that $G_0 \actsOnPoint N \subseteq N$, and let $\bullet_{H_0}$ be the left group action $\actsOnMap\restrictedTo_{H_0 \times X_N \to X_N}$ of $H_0$ on $X_N$. The map $\Delta$ is called \define{$\kappa$-local}\graffito{$\kappa$-local} if and only if there is a $\bullet_{H_0}$-invariant map $\delta \from X_N \to Q$ such that 
    \begin{equation*}
      \ForEach x \in X \ForEach m \in M \Holds \Delta(x)(m) = \delta(n \mapsto x(m \isSemiActedUponBy n)). 
    \end{equation*}
  \end{definition}

  \begin{remark} 
    For each point $x \in X$ and each cell $m \in M$, we have $\Delta(x)(m) = \delta((g_{m_0, m}^{-1} \actsOnMap x)\restrictedTo_N)$.
  \end{remark}

  \begin{remark}
  \label{rem:k-local-for-greater-k}
    Let $\Delta$ be a $\kappa$-local map from $X$ to $Y$. For each non-negative integer $\kappa'$ such that $\kappa' \geq \kappa$, the map $\Delta$ is $\kappa'$-local. 
  \end{remark}

  \begin{definition}
  \label{def:local-map}
    Let $X$ and $Y$ be two subshifts of $Q^M$ and let $\Delta$ be a map from $X$ to $Y$. The map $\Delta$ is called \define{local}\graffito{local} if and only if there is a non-negative integer $\kappa$ such that it is $\kappa$-local.
  \end{definition}

  \begin{remark}
  \label{rem:local-map-iff-cellular-automaton}
    Let $X$ and $Y$ be two subshifts of $Q^M$ and let $\Delta$ be a map from $X$ to $Y$. The map $\Delta$ is local if and only if it is the restriction to $X \to Y$ of the global transition function of a big-cellular automaton over $\mathcal{R}$ with set of states $Q$ and finite neighbourhood. 
  \end{remark}

  \begin{example}[{Sliding Block Codes \cite[Definition~1.5.1]{lind:marcus:1995}}]
    In the situation of \cref{ex:full-shift-over-integers}, local maps are but sliding block codes.
  \end{example}

  \begin{example}[{From Golden Mean to Even Shift \cite[Example~1.5.6]{lind:marcus:1995}}]
  \label{ex:from-golden-mean-to-even-shift-local-map}
    Let $X$ be the golden mean shift (\cref{ex:shift:golden-mean}), let $Y$ be the even shift (\cref{ex:shift:even}), let $\delta$ be the map from $X_{\setOf{0, 1}}$ to $\setOf{0, 1}$ given by $00 \mapsto 1$, $01 \mapsto 0$, and $10 \mapsto 0$, and let $\Delta$ be the map from $X$ to $Y$ given by $x \mapsto [z \mapsto \delta(n \mapsto x(z + n))]$. The map $\Delta$ is, by definition, local and it is, according to \cite[Example~1.5.6]{lind:marcus:1995}, surjective.
  \end{example}

  Domain and codomain of a local map can be restricted simultaneously to a subset of cells and its interior, as is done in

  \begin{definition} 
  \label{def:restriction-of-local-map}
    Let $\Delta$ be a $\kappa$-local map from $X$ to $Y$ and let $A$ be a subset of $M$. The map
    \begin{align*}
      \Delta_A^- \from X_A &\to     Y_{A^{-\kappa}},\\
                         p &\mapsto \Delta(c)\restrictedTo_{A^{-\kappa}}, \text{ where } c \in X \text{ such that } c\restrictedTo_A = p,
    \end{align*}
    is called \define{restriction of $\Delta$ to $A$}\graffito{restriction $\Delta_A^-$ of $\Delta$ to $A$}.
  \end{definition}

  The image of a local map is a shift space, which is shown in

  \begin{lemma} 
  \label{lem:image-of-local-map-is-subshift}
    Let $\Delta$ be a local map from $X$ to $Y$. Its image $\Delta(X)$ is a subshift of $Q^M$.
  \end{lemma}

  \begin{proof} 
    According to \cref{lem:shift-space-iff-invariant-and-compact}, the shift space $X$ is compact. And, according to \cref{rem:local-map-iff-cellular-automaton} and \cite[Corollary~3]{wacker:automata:2016}, the map $\Delta$ is continuous. Therefore, the topological space $\Delta(X)$ is compact and hence closed. Moreover, because $h \actsOnMap \Delta(X) = \Delta(h \actsOnMap X) = \Delta(X)$, the topological space $\Delta(X)$ is shift-invariant. Therefore, according to \cref{lem:shift-space-iff-invariant-and-compact}, the topological space $\Delta(X)$ is a subshift of $Q^M$.
  \end{proof}

  The difference of two points of the full shift is the set of cells in which they differ, as defined in

  \begin{definition}
  \label{def:difference-of-two-points-of-the-full-shift}
    Let $c$ and $c'$ be two points of $Q^M$. The set
      $\differenceOf(c, c') = \setOf{m \in M \suchThat c(m) \neq c'(m)}$ 
    is called \define{difference of $c$ and $c'$}.
  \end{definition}

  A local map is pre-injective if it is injective on points with finite support, as defined in

  \begin{definition}
  \label{def:pre-injective-for-local-maps}
    Let $\Delta$ be a local map from $X$ to $Y$. It is called \define{pre-injective}\graffito{pre-injective} if and only if, for each tuple $(x, x') \in X \times X$ such that $\differenceOf(x, x')$ is finite and $\Delta(x) = \Delta(x')$, we have $x = x'$.
  \end{definition}

  A bijective local map with local inverse is a conjugacy, and its domain and codomain are conjugate, which is defined in

  \begin{definition} 
  \label{def:conjugacy}
    Let $\Delta$ be a local map from $X$ to $Y$. It is called \define{conjugacy}\graffito{conjugacy} if and only if it is bijective and its inverse is local.
  \end{definition}

  \begin{definition} 
  \label{def:conjugate}
    Let $X$ and $Y$ be two shift spaces. They are called \define{conjugate}\graffito{conjugate} if and only if there is a conjugacy from $X$ to $Y$.
  \end{definition}

  Entropy of shift spaces is invariant under conjugacy, which is shown in

  \begin{lemma}
  \label{lem:entropy-invariant-under-conjugacy}
    Let $\mathcal{R}$ be right amenable, let $\mathcal{F}$ be a right Følner net in $\mathcal{R}$, and let $X$ and $Y$ be two conjugate subshifts of $Q^M$. Then, $\entropyOf_{\mathcal{F}}(X) = \entropyOf_{\mathcal{F}}(Y)$.
  \end{lemma}

  \begin{proof}
    This is a direct consequence of \cref{rem:local-map-iff-cellular-automaton} and \cite[Theorem~3]{wacker:garden:2016}.
  \end{proof}


  \begin{definition} 
  \label{def:Moore-and-Myhill-properties}
    Let $X$ be a subshift of $Q^M$. It is said to have the 
    \begin{enumerate}
      \item \define{Moore property}\graffito{Moore property} if and only if each surjective local map from $X$ to $X$ is pre-injective;
      \item \define{Myhill property}\graffito{Myhill property} if and only if each pre-injective local map from $X$ to $X$ is surjective. 
    \end{enumerate}
  \end{definition}

  \begin{remark}
  \label{rem:Moore-and-Myhill-are-invariant-under-conjugacy}
    Both the Moore and the Myhill property is invariant under conjugacy.
  \end{remark}

  \section{Tilings}
  \label{sec:tilings} 

  \subsubsection*{Contents.} A $\ntuple{\theta, \kappa, \theta'}$-tiling is a subset of cells such that the balls of radius $\theta$ about those cells are pairwise at least $\kappa + 1$ apart and the balls of radius $\theta'$ about those cells cover all cells (see \cref{def:pairwise-at-least-apart,def:theta-kappa-theta-prime-tiling}). If there are infinitely many cells, then, for each $\theta$ and $\kappa$, there is a $\ntuple{\theta, \kappa, 4 \theta + 2 \kappa}$-tiling (see \cref{thm:there-is-a-rho-kappa-theta-tiling}). And, a subset of a strongly irreducible shift space has less entropy than that space if about each point of a $\ntuple{\theta, \kappa, \theta'}$-tiling the subset has fewer patterns with ball-shaped domains of radius $\theta$ than the space (see \cref{thm:entropy-less-if-real-subsets-for-rho-kappa-theta-tiling}); in the proof of that statement we use \cref{lem:technical-inequality-for-theorem-entropy-less-if-real-subsets-for-rho-kappa-theta-tiling,lem:set-contained-in-unition-of-balls-and-of-interior-of-set,cor:subshifts-upper-bound-of-tiling-cap-folner-net,lem:subshift-with-at-least-two-points-has-at-least-two-patterns-for-each-domain}.

  \begin{definition}
  \label{def:pairwise-at-least-apart}
    Let $\family{A_j}_{j \in J}$ be a family of subsets of $M$ and let $\kappa$ be a non-negative integer. The family $\family{A_j}_{j \in J}$ is called \define{pairwise at least $\kappa + 1$ apart}\graffito{pairwise at least $\kappa + 1$ apart} if and only if 
    \begin{equation*}
      \ForEach j \in J \ForEach j' \in J \Holds \parens[\big]{j \neq j' \implies d(A_j, A_{j'}) \geq \kappa + 1}. 
    \end{equation*}
  \end{definition}

  \begin{remark}
  \label{rem:pairwise-kappa-apart-vs-pairwise-disjoint}
    Each pairwise at least $\kappa + 1$ apart family is pairwise disjoint. And each pairwise disjoint family is pairwise at least $0 + 1$ apart.
  \end{remark}

  \begin{definition} 
  \label{def:theta-kappa-theta-prime-tiling}
    Let $T$ be a subset of $M$, and let $\theta$, $\kappa$, and $\theta'$ be three non-negative integers. The set $T$ is called \define{$\ntuple{\theta, \kappa, \theta'}$-tiling of $\mathcal{R}$}\graffito{$\ntuple{\theta, \kappa, \theta'}$-tiling $T$ of $\mathcal{R}$} if and only if the family $\family{\ball(t, \theta)}_{t \in T}$ is pairwise at least $\kappa + 1$ apart and the family $\family{\ball(t, \theta')}_{t \in T}$ is a cover of $M$.
  \end{definition}

  \begin{remark}
  \label{rem:apart-tilint-vs-tiling}
    According to \cref{rem:pairwise-kappa-apart-vs-pairwise-disjoint}, each $\ntuple{\theta, \kappa, \theta'}$-tiling of $\mathcal{R}$ is a $\ntuple{\ball(\theta), \ball(\theta')}$-tiling of $\mathcal{R}$, see \cite[Definition 2]{wacker:garden:2016}; and each $\ntuple{\ball(\theta), \ball(\theta')}$-tiling of $\mathcal{R}$ is a $\ntuple{\theta, 0, \theta'}$-tiling of $\mathcal{R}$.
  \end{remark}

%

  Greedily picking elements that are pairwise far enough apart yields a tiling, which we show in

  \begin{theorem} 
  \label{thm:there-is-a-rho-kappa-theta-tiling}
    Let $M$ be infinite, and let $\theta$ and $\kappa$ be two non-negative integers. There is a countably infinite $\ntuple{\theta, \kappa, \theta'}$-tiling $T$ of $\mathcal{R}$, where $\theta' = 4 \theta + 2 \kappa$.
  \end{theorem}

  \begin{usage-note}
    In the proof, infiniteness of $M$ is used to deduce that spheres of arbitrarily big radii are non-empty.
  \end{usage-note}

  \begin{proof-sketch}
    From each of the spheres $\sphere(i (2 \theta + \kappa + 1))$, for $i \in \N_0$, pick elements that are pairwise at least $2 \theta + \kappa + 1$ apart and whose $(2 \theta + \kappa)$-closure covers the sphere --- they constitute a set $T$. The family $\family{\ball(t, \theta)}_{t \in T}$ is pairwise at least $\kappa + 1$ apart and the family $\family{\ball(t, 4 \theta + 2 \kappa)}_{t \in T}$ is a cover of $M$. See \cref{fig:there-is-a-rho-kappa-theta-tiling} for a schematic representation.
    \begin{figure}
      \centering
      \begin{tikzpicture}[> = To, x=1cm, y=1cm, circle dotted/.style = {dash pattern = on .05mm off 1.5pt, line cap = round}]
        \pgfmathsetmacro\t{0.5} 
        \pgfmathsetmacro\k{1} 
        \pgfmathsetmacro\c{5} 
        \pgfmathsetmacro\r{0.0625} 

        \draw[fill = black] (0,0) circle[radius = \r] node {}; 
        \draw (0,0) circle[radius = {1 * ((2 * \t) + \k + 1)}] node {}; 
        \draw (0,0) circle[radius = {2 * ((2 * \t) + \k + 1)}] node {}; 

        \foreach \a in {1,2,...,\c} {
          \draw[fill = black] (\a * 360 / \c: {1 * ((2 * \t) + \k + 1)}) circle[radius = \r] node[below] {$m_{i, \a}$};
          \draw[circle dotted] (\a * 360 / \c: {1 * ((2 * \t) + \k + 1)}) circle[radius = \t];
          \draw[dashdotted] (\a * 360 / \c: {1 * ((2 * \t) + \k + 1)}) circle[radius = {\t + \k}];
          \draw[dashed] (\a * 360 / \c: {1 * ((2 * \t) + \k + 1)}) circle[radius = {(2 * \t) + \k}];
        }
        \draw[fill = black] (1.75 * 360 / \c: {1.75 * ((2 * \t) + \k + 1)}) circle[radius = \r] node[above] {$m$}
                            -- 
                            (1.75 * 360 / \c: {1 * ((2 * \t) + \k + 1)}) circle[radius = \r] node[below] {$m'$}
                            -- 
                            (2 * 360 / \c: {1 * ((2 * \t) + \k + 1)}) node[below] {$\phantom{m_{i, 2}}\ \qquad = m''$};
      \end{tikzpicture}
      \caption{%
        Schematic representation of the set-up of the proof of \cref{thm:there-is-a-rho-kappa-theta-tiling}. 
        The whole space is $M$; the dot in the centre is $m_0$; the smaller solid circle is $M_{i, 1} = \sphere(i (2 \theta + \kappa + 1))$ and the larger solid circle is $M_{i + 1, 1} = \sphere((i + 1) (2 \theta + \kappa + 1))$; the dots on the smaller solid circle are the elements of $M_i$; the region enclosed by the dotted circle about $m_{i,j}$ is $\ball(m_{i,j}, \theta)$; the region enclosed by the dash-dotted circle about $m_{i,j}$ is $\ball(m_{i,j}, \theta + \kappa)$; the region enclosed by the dashed circle about $m_{i,j}$ is $\ball(m_{i,j}, 2 \theta + \kappa)$; the dots labelled $m$, $m'$, and $m''$ are the respective elements from the last part of the proof.
      }
      \label{fig:there-is-a-rho-kappa-theta-tiling}
    \end{figure}
  \end{proof-sketch}

  \begin{proof} 
    Let $i \in \N_0$. Furthermore, let $M_{i,1} = \sphere(i (2 \theta + \kappa + 1))$. Then, because $M$ is infinite, according to \cite[Corollary~8.15]{wacker:growth:2017}, the set $M_{i,1}$ is non-empty and finite. For $j \in \N_+$ in increasing order, if $M_{i,j} \smallsetminus \setOf{m_{i,1}, m_{i,2}, \dotsc, m_{i,j-1}} \neq \emptyset$, then choose $m_{i,j} \in M_{i,j} \smallsetminus \setOf{m_{i,1}, m_{i,2}, \dotsc, m_{i,j-1}}$ and put
    \begin{align*}
      M_{i,j+1} &= \setOf{m_{i,j}} \cup M_{i,j} \smallsetminus \ball(m_{i,j}, 2 \theta + \kappa)\\
                &= \setOf{m \in M_{i,j} \suchThat d(m, m_{i,j}) = 0 \text{ or } d(m, m_{i,j}) \geq 2 \theta + \kappa + 1};
    \end{align*}
    otherwise stop, put $j_i = j$ and put $M_i = M_{i, j_i}$ (see \cref{fig:there-is-a-rho-kappa-theta-tiling}).

    By construction,
    \begin{equation}
    \label{eq:there-is-a-rho-kappa-theta-tiling:apart}
      \ForEach m \in M_i \ForEach m' \in M_i \Holds \parens[\big]{m \neq m' \implies d(m, m') \geq 2 \theta + \kappa + 1},
    \end{equation}
    and
    \begin{equation}
    \label{eq:there-is-a-rho-kappa-theta-tiling:base-point}
      \ForEach m \in M_{i,1} \Exists m' \in M_i \SuchThat d(m, m') \leq 2 \theta + \kappa,
    \end{equation}
    and, for each $i' \in \N_0$ with $i' \neq i$, because $M_i \subseteq M_{i,1}$ and $M_{i'} \subseteq M_{i',1}$, and $M$ is infinite, according to \cite[Corollary~5.18]{wacker:growth:2017},
    \begin{equation}
    \label{eq:there-is-a-rho-kappa-theta-tiling:sets-apart}
      d(M_i, M_{i'}) \geq d(M_{i,1}, M_{i',1}) \geq 2 \theta + \kappa + 1.
    \end{equation}

    Let $T = \bigcup_{i \in \N_0} M_i$. Because, for each $i \in \N_0$, the set $M_i$ is finite, the set $T$ is countable. And, because $\sequence{M_{i,1}}_{i \in \N_0}$ is pairwise disjoint, so is $\sequence{M_i}_{i \in \N_0}$ and hence $T$ is infinite. 
    Thus, $T$ is countably infinite.

    \proofpart{Subproof of: $\family{\ball(t, \theta)}_{t \in T}$ is pairwise at least $\kappa + 1$ apart}
    Let $t$, $t' \in T$ such that $t \neq t'$. If there is an $i \in \N_0$ such that $t$, $t' \in M_i$, then, according to \cref{eq:there-is-a-rho-kappa-theta-tiling:apart}, we have $d(t, t') \geq 2 \theta + \kappa + 1$. Otherwise, there are $i$, $i' \in \N_0$ with $i \neq i'$ such that $t \in M_i$ and $t' \in M_{i'}$, and then, according to \cref{eq:there-is-a-rho-kappa-theta-tiling:sets-apart}, we have $d(t, t') \geq d(M_i, M_{i'}) \geq 2 \theta + \kappa + 1$. In conclusion, in both cases, according to \cite[Lemma~5.20]{wacker:growth:2017}, we have $d(\ball(t, \theta), \ball(t', \theta)) \geq \kappa + 1$.

    \proofpart{Subproof of: $\family{\ball(t, 4 \theta + 2 \kappa)}_{t \in T}$ is a cover of $M$ (see \cref{fig:there-is-a-rho-kappa-theta-tiling})}
    Let $m \in M$. Then, there is an $i \in \N_0$ such that $i (2 \theta + \kappa + 1) \leq \absoluteValueOf{m} < (i + 1) (2 \theta + \kappa + 1)$. Hence, according to \cite[Lemma~5.17]{wacker:growth:2017},
    \begin{multline*}
      d(m, M_{i,1}) = d(M_{i,1}, m) \leq d(m_0, m) - i (2 \theta + \kappa + 1)\\
          < (i + 1) (2 \theta + \kappa + 1) - i (2 \theta + \kappa + 1) = 2 \theta + \kappa + 1.
    \end{multline*}
    Thus, $d(m, M_{i,1}) \leq 2 \theta + \kappa$. Therefore, there is an $m' \in M_{i,1}$ such that $d(m, m') \leq 2 \theta + \kappa$. Moreover, according to \cref{eq:there-is-a-rho-kappa-theta-tiling:base-point}, there is an $m'' \in M_i$ such that $d(m', m'') \leq 2 \theta + \kappa$. Thus,
    \begin{equation*}
      d(m, m'') \leq d(m, m') + d(m', m'') \leq 4 \theta + 2 \kappa.
    \end{equation*}
    Hence, $m \in \ball(m'', 4 \theta + 2 \kappa)$. Therefore, because $m'' \in T$, we have $m \in \bigcup_{t \in T} \ball(t, 4 \theta + 2 \kappa)$. In conclusion, $\bigcup_{t \in T} \ball(t, 4 \theta + 2 \kappa) = M$. 
  \end{proof}

  One by one forbidding subpatterns with similar domains that are far enough apart in a set of finite patterns, decreases its size by at least a multiplicative constant between $0$ and $1$ in each step. In other words, the number of finite patterns with a fixed domain, excluding those in which some subpatterns with similar domains that are far enough apart are embedded, is bounded above by some constant between $0$ and $1$ raised to the power of the number of forbidden subpatterns times the number of all finite patterns with the fixed domain, which is shown in 

  \begin{lemma} 
  \label{lem:technical-inequality-for-theorem-entropy-less-if-real-subsets-for-rho-kappa-theta-tiling}
    Let $X$ be a non-empty and $\kappa$-strongly irreducible subshift of $Q^M$, let $F$ be a finite subset of $M$, let $\theta$ be a non-negative integer, let $T$ be a subset of $M$ such that the family $\family{\ball(t, \theta)}_{t \in T}$ is pairwise at least $\kappa + 1$ apart, and, for each element $t \in T$, let $p_t$ be a pattern of $X_{\ball(t, \theta)}$. Furthermore, let $\xi$ be the positive integer $\absoluteValueOf{X_{\ball(\theta)^{+\kappa}}}$, let $S$ be the finite set $T \cap F^{-(\theta + \kappa)}\ (= \setOf{t \in T \suchThat \ball(t, \theta)^{+\kappa} \subseteq F})$, and, for each element $s \in S$, let $\pi_s$ be the map $X_F \to X_{\ball(s, \theta)}$, $p \mapsto p\restrictedTo_{\ball(s, \theta)}$ (see \cref{fig:technical-inequality-for-theorem-entropy-less-if-real-subsets-for-rho-kappa-theta-tiling}).
    \begin{figure}
      \centering
      \begin{minipage}[c]{\textwidth / 2} 
        \begin{tikzpicture}[circle dotted/.style = {dash pattern = on .05mm off 1.5pt, line cap = round}] 
          \pgfmathsetmacro\t{0.3} 
          \pgfmathsetmacro\k{0.1} 
          \pgfmathsetmacro\xFi{1 - 0.5} 
          \pgfmathsetmacro\yFi{1 - 0.5} 
          \pgfmathsetmacro\wFi{2.7} 
          \pgfmathsetmacro\hFi{3} 

          \foreach \x in {0, ..., 4} {
            \foreach \y in {0, ..., 4} {
              \path[fill] (\x, \y) circle (1pt); 
              \draw (\x - \t, \y - \t) rectangle (\x + \t, \y + \t); 
              \draw[dashdotted] (\x - \t - \k, \y - \t - \k) rectangle (\x + \t + \k, \y + \t + \k); 
            }
          }
          \draw[dashed, pattern = north east lines] (\xFi, \yFi) rectangle (\xFi + \wFi, \yFi + \hFi); 
          \foreach \x in {1, ..., 2} {
            \foreach \y in {1, ..., 3} {
              \draw[fill = white] (\x - \t, \y - \t) rectangle (\x + \t, \y + \t); 
              \draw (\x, \y) circle (1pt); 
            }
          }
          \draw[line width = 0.75pt, circle dotted] (\xFi + \t + \k, \yFi + \t + \k) rectangle (\xFi + \wFi - \t - \k, \yFi + \hFi - \t - \k); 
        \end{tikzpicture}
      \end{minipage}%
      \begin{minipage}[c]{\textwidth / 2} 
        The whole space is $M$; the dots and circles are the elements of the set $T$; for each element $t \in T$, the region enclosed by the rectangle with solid border about $t$ is the set $\ball(t, \theta)$ and the region enclosed by the rectangle with dash-dotted border about $t$ is the set $\ball(t, \theta)^{+\kappa}$; the region enclosed by the rectangle with dashed border is $F$; the region enclosed by the rectangle with dotted border is $F^{-(\theta + \kappa)}$; the circles are the elements of $S = T \cap F^{-(\theta + \kappa)}$; the hatched region is the set $F \smallsetminus (\bigcup_{s \in S} \ball(s, \theta))$. 
      \end{minipage}
      \caption{Schematic representation of the set-up of \cref{lem:technical-inequality-for-theorem-entropy-less-if-real-subsets-for-rho-kappa-theta-tiling}.}
      \label{fig:technical-inequality-for-theorem-entropy-less-if-real-subsets-for-rho-kappa-theta-tiling}
    \end{figure}
    Then, 
    \begin{equation*}
      \absoluteValueOf{X_F \smallsetminus \bigcup_{s \in S} \pi_s^{-1}(p_s)}
      \leq
      (1 - \xi^{-1})^{\absoluteValueOf{S}} \cdot \absoluteValueOf{X_F}. 
    \end{equation*}
  \end{lemma}

  \begin{usage-note}
    In the proof, $\kappa$-strong irreducibility is used to extend an in $X$ allowed $F \smallsetminus \ball(s, \theta)^{+\kappa}$-pattern by the $\ball(s, \theta)$-pattern $p_s$ and a $\boundaryOf_\kappa^+ \ball(s, \theta)$-pattern to an in $X$ allowed $F$-pattern. 
  \end{usage-note}

  \begin{proof-sketch}
    Let $\family{s_j}_{j \in \setOf{1, 2, \dotsc, \absoluteValueOf{S}}}$ be an enumeration of $S$, let $Z_0 = X_F$, and, for each $\vartheta \in \setOf{0, 1, \dotsc, \absoluteValueOf{S} - 1}$, let $Z_{\vartheta + 1} = Z_\vartheta \smallsetminus \parens{\pi_{s_{\vartheta + 1}}^{-1}(p_{s_{\vartheta + 1}}) \cap Z_\vartheta}$. Furthermore, let $\vartheta \in \setOf{0, 1, \dotsc, \absoluteValueOf{S} - 1}$. Then, $\absoluteValueOf{Z_\vartheta} \leq \absoluteValueOf{X_{\ball(s_{\vartheta + 1}, \theta)^{+ \kappa}}} \cdot \absoluteValueOf{(Z_\vartheta)_{F \smallsetminus \ball(s_{\vartheta + 1}, \theta)^{+ \kappa}}} = \xi \cdot \absoluteValueOf{(Z_\vartheta)_{F \smallsetminus \ball(s_{\vartheta + 1}, \theta)^{+ \kappa}}}$. And, because $X$ is $\kappa$-strongly irreducible, each pattern of $(Z_\vartheta)_{F \smallsetminus \ball(s_{\vartheta + 1}, \theta)^{+ \kappa}}$ can be extended by $p_{s_{\vartheta + 1}}$ and a pattern with domain $\ball(s_{\vartheta + 1}, \theta)^{+ \kappa} \smallsetminus \ball(s_{\vartheta + 1}, \theta)$ to a pattern of $\pi_{s_{\vartheta + 1}}^{-1}(p_{s_{\vartheta + 1}}) \cap Z_\vartheta$ and thus $\absoluteValueOf{(Z_\vartheta)_{F \smallsetminus \ball(s_{\vartheta + 1}, \theta)^{+ \kappa}}} \leq \absoluteValueOf{\pi_{s_{\vartheta + 1}}^{-1}(p_{s_{\vartheta + 1}}) \cap Z_\vartheta}$. Hence, $\absoluteValueOf{\pi_{s_{\vartheta + 1}}^{-1}(p_{s_{\vartheta + 1}}) \cap Z_\vartheta} \geq \xi^{-1} \cdot \absoluteValueOf{Z_\vartheta}$. Therefore,
      $\absoluteValueOf{Z_{\vartheta + 1}}
      =    \absoluteValueOf{Z_\vartheta} - \absoluteValueOf{\pi_{s_{\vartheta + 1}}^{-1}(p_{s_{\vartheta + 1}}) \cap Z_\vartheta}
      \leq (1 - \xi^{-1}) \cdot \absoluteValueOf{Z_\vartheta}$.
    The statement follows by induction.
  \end{proof-sketch}

  \begin{proof}
    As claimed, because $X \neq \emptyset$, the integer $\xi$ is positive; because, according to \cite[Corollary~5.14]{wacker:growth:2017} and \cite[Item~2 of Corollary~6.3]{wacker:growth:2017}, $s \isSemiActedUponBy \ball(\theta + \kappa) = \ball(s, \theta)^{+\kappa}$, we have $S = T \cap F^{-(\theta + \kappa)} = \setOf{t \in T \suchThat \ball(t, \theta)^{+\kappa} \subseteq F}$; and, because $S \subseteq \bigcup_{s \in S} \ball(s, \theta)^{+\kappa} \subseteq F$ and $F$ is finite, the set $S$ is finite.

    Let $\family{B_s}_{s \in S} = \family{\ball(s, \theta)}_{s \in S}$, let $\family{s_j}_{j \in \setOf{1, 2, \dotsc, \absoluteValueOf{S}}}$ be an enumeration of $S$, 
    and, for each $\vartheta \in \setOf{0, 1, \dotsc, \absoluteValueOf{S}}$, let
    \begin{equation*}
      Z_\vartheta = X_F \smallsetminus \bigcup_{j = 1}^\vartheta \pi_{s_j}^{-1}(p_{s_j})
       \ \parens*{= \setOf{p \in X_F \suchThat \ForEach j \in \setOf{1, 2, \dotsc, \vartheta} \Holds p\restrictedTo_{B_{s_j}} \neq p_{s_j}}}.
    \end{equation*}
    To establish the claim, we prove by induction on $\vartheta$, that, for each $\vartheta \in \setOf{0, 1, \dotsc, \absoluteValueOf{S}}$,
    \begin{equation}
    \label{eq:entropy-less-if-real-subsets-for-rho-kappa-theta-tiling:inductive-hypothesis}
      \absoluteValueOf{Z_\vartheta}
      \leq
      (1 - \xi^{-1})^\vartheta \cdot \absoluteValueOf{X_F}.
    \end{equation}

    \proofpart{Base Case}
      Let $\vartheta = 0$. Then, because $\bigcup_{j = 1}^0 \pi_{s_j}^{-1}(p_{s_j}) = \emptyset$ and $(1 - \xi^{-1})^0 = 1$, \cref{eq:entropy-less-if-real-subsets-for-rho-kappa-theta-tiling:inductive-hypothesis} holds. Note that $0^0 = 1$.

    \proofpart{Inductive Step} 
      Let $\vartheta \in \setOf{0, 1, \dotsc, \absoluteValueOf{S} - 1}$ such that \cref{eq:entropy-less-if-real-subsets-for-rho-kappa-theta-tiling:inductive-hypothesis}, called \emph{inductive hypothesis}, holds. Furthermore, let $Z = Z_\vartheta$. Because $B_{s_{\vartheta + 1}}^{+\kappa} \subseteq F$, 
      \begin{equation*}
        Z \subseteq Z_{B_{s_{\vartheta + 1}}^{+\kappa}} \times Z_{F \smallsetminus B_{s_{\vartheta + 1}}^{+\kappa}}
          \subseteq X_{B_{s_{\vartheta + 1}}^{+\kappa}} \times Z_{F \smallsetminus B_{s_{\vartheta + 1}}^{+\kappa}}.
      \end{equation*}
      Hence, $\absoluteValueOf{Z} \leq \absoluteValueOf{X_{B_{s_{\vartheta + 1}}^{+\kappa}}} \cdot \absoluteValueOf{Z_{F \smallsetminus B_{s_{\vartheta + 1}}^{+\kappa}}}$. 
      Moreover, according to \cref{rem:pattern-belongs-to-shift-if-and-only-if-translated-pattern-belongs}, we have $\absoluteValueOf{X_{B_{s_{\vartheta + 1}}^{+\kappa}}} = \absoluteValueOf{X_{\ball(\theta)^{+\kappa}}} = \xi$ (where we used that $B_{s_{\vartheta + 1}}^{+\kappa} = s_{\vartheta + 1} \isSemiActedUponBy \ball(\theta)^{+\kappa}$, which holds according to \cite[Corollary~5.14]{wacker:growth:2017} and \cite[Item~2 of Corollary~6.3]{wacker:growth:2017}). Therefore, $\absoluteValueOf{Z} \leq \xi \cdot \absoluteValueOf{Z_{F \smallsetminus B_{s_{\vartheta + 1}}^{+\kappa}}}$.

      Let $p \in Z_{F \smallsetminus B_{s_{\vartheta + 1}}^{+\kappa}}$. Then, $p \in X_{F \smallsetminus B_{s_{\vartheta + 1}}^{+\kappa}}$. Moreover, according to \cite[Lemma~6.5]{wacker:growth:2017}, we have $d(B_{s_{\vartheta + 1}}, F \smallsetminus B_{s_{\vartheta + 1}}^{+\kappa}) \geq \kappa + 1$. Hence, because $X$ is $\kappa$-strongly irreducible, there is a $p'' \in X_F$ such that $p''\restrictedTo_{\domainOf(p)} = p$ and $p''\restrictedTo_{\domainOf(p_{s_{\vartheta + 1}})} = p_{s_{\vartheta + 1}}$. 
      Furthermore, because $\family{\ball(t, \theta)}_{t \in T}$ is pairwise at least $\kappa + 1$ apart, for each $j \in \setOf{1, 2, \dotsc, \vartheta}$, we have $B_{s_j} \subseteq F \smallsetminus B_{s_{\vartheta + 1}}^{+\kappa}$. Therefore, for each $j \in \setOf{1, 2, \dotsc, \vartheta}$, we have $p''\restrictedTo_{B_{s_j}} = p\restrictedTo_{B_{s_j}} \neq p_{s_j}$ and hence $p'' \notin \pi_{s_j}^{-1}(p_{s_j})$. Thus, $p'' \in Z$. Moreover, $p'' \in \pi_{s_{\vartheta + 1}}^{-1}(p_{s_{\vartheta + 1}})$. Therefore, $\absoluteValueOf{Z_{F \smallsetminus B_{s_{\vartheta + 1}}^{+\kappa}}} \leq \absoluteValueOf{\pi_{s_{\vartheta + 1}}^{-1}(p_{s_{\vartheta + 1}}) \cap Z}$. 

      Together,
      \begin{equation*}
        \absoluteValueOf{Z} \leq \xi \cdot \absoluteValueOf{\pi_{s_{\vartheta + 1}}^{-1}(p_{s_{\vartheta + 1}}) \cap Z}.
      \end{equation*}
      Because $Z_{\vartheta + 1} = Z \smallsetminus \pi_{s_{\vartheta + 1}}^{-1}(p_{s_{\vartheta + 1}})$,
      \begin{align*}
        \absoluteValueOf{Z_{\vartheta + 1}}
        &=    \absoluteValueOf{Z \smallsetminus \pi_{s_{\vartheta + 1}}^{-1}(p_{s_{\vartheta + 1}})}\\
        &=    \absoluteValueOf{Z \smallsetminus (\pi_{s_{\vartheta + 1}}^{-1}(p_{s_{\vartheta + 1}}) \cap Z)}\\ 
        &=    \absoluteValueOf{Z} - \absoluteValueOf{\pi_{s_{\vartheta + 1}}^{-1}(p_{s_{\vartheta + 1}}) \cap Z}\\
        &\leq \absoluteValueOf{Z} - \xi^{-1} \cdot \absoluteValueOf{Z}\\
        &=    (1 - \xi^{-1}) \cdot \absoluteValueOf{Z}.
      \end{align*}
      Hence, according to the inductive hypothesis,
      \begin{align*}
        \absoluteValueOf{Z_{\vartheta + 1}}
        &\leq (1 - \xi^{-1}) \cdot (1 - \xi^{-1})^\vartheta \cdot \absoluteValueOf{X_F}\\
        &=    (1 - \xi^{-1})^{\vartheta + 1} \cdot \absoluteValueOf{X_F}.
      \end{align*}

    In conclusion, according to the principle of mathematical induction, for each $\vartheta \in \setOf{0, 1, \dotsc, \absoluteValueOf{S}}$, \cref{eq:entropy-less-if-real-subsets-for-rho-kappa-theta-tiling:inductive-hypothesis} holds.
  \end{proof}

  The number of elements in a finite set is bounded above by the number of elements its interior shares with a tiling times the number of elements of a big enough ball plus the number of elements of a big enough boundary of the finite set, which is shown in

  \begin{lemma} 
  \label{lem:set-contained-in-unition-of-balls-and-of-interior-of-set}
    Let $F$ be a finite subset of $M$, let $\theta$, $\kappa$, and $\theta'$ be three non-negative integers, let $T$ be a subset of $M$ such that $\family{\ball(t, \theta')}_{t \in T}$ is a cover of $M$, and let $S$ be the finite set $T \cap F^{-(\theta + \kappa)}\ (= \setOf{t \in T \suchThat \ball(t, \theta)^{+\kappa} \subseteq F})$ (see \cref{fig:set-contained-in-unition-of-balls-and-of-interior-of-set}).
    \begin{figure}
      \centering
      \begin{minipage}[c]{\textwidth / 2} 
        \begin{tikzpicture}[circle dotted/.style = {dash pattern = on .05mm off 1.5pt, line cap = round}] 
          \pgfmathsetmacro\t{0.4} 
          \pgfmathsetmacro\tp{0.7} 
          \pgfmathsetmacro\xF{1 - 0.55} 
          \pgfmathsetmacro\yF{1 - 0.55} 
          \pgfmathsetmacro\wF{2.8} 
          \pgfmathsetmacro\hF{3.1} 

          \foreach \x in {1, ..., 2} {
            \foreach \y in {1, ..., 3} {
              \draw (\x, \y) circle (1pt);
            }
          }
          \foreach \x in {0, 3, 4} {
            \foreach \y in {0, ..., 4} {
              \fill (\x, \y) circle (1pt);
            }
          }
          \foreach \x in {1, 2} {
            \foreach \y in {0, 4} {
              \fill (\x, \y) circle (1pt);
            }
          }
          \foreach \x in {0, ..., 4} {
            \foreach \y in {0, ..., 4} {
              \draw (\x - \t, \y - \t) rectangle (\x + \t, \y + \t); 
              \draw[gray, dashdotted] (\x - \tp, \y - \tp) rectangle (\x + \tp, \y + \tp); 
            } 
          }
          \draw[line width = 0.5pt, dashed] (\xF, \yF) rectangle (\xF + \wF, \yF + \hF); 
          \draw[line width = 0.75pt, circle dotted] (\xF + \t, \yF + \t) rectangle (\xF + \wF - \t, \yF + \hF - \t); 
          \draw[line width = 0.75pt, circle dotted] (\xF + \t + \tp, \yF + \t + \tp) rectangle (\xF + \wF - \t - \tp, \yF + \hF - \t - \tp); 
        \end{tikzpicture}
      \end{minipage}%
      \begin{minipage}[c]{\textwidth / 2}
        The whole space is $M$; the dots and circles are the elements of the set $T$; for each element $t \in T$, the region enclosed by the rectangle with solid border about $t$ is the set $\ball(t, \theta)^{+\kappa}$ and the region enclosed by the rectangle with dash-dotted border about $t$ is the set $\ball(t, \theta')$; the region enclosed by the rectangle with dashed border is $F$; the region enclosed by the smallest rectangle with dotted border is $F^{-(\theta + \kappa + \theta')}$ and the region enclosed by the largest rectangle with dotted border is $F^{-(\theta + \kappa)}$; the circles are the elements of $S = T \cap F^{-(\theta + \kappa)}$.
      \end{minipage}
      \caption{Schematic representation of the set-up of \cref{lem:set-contained-in-unition-of-balls-and-of-interior-of-set}.}
      \label{fig:set-contained-in-unition-of-balls-and-of-interior-of-set}
    \end{figure}
    Then,
    \begin{equation*}
      \absoluteValueOf{F} \leq \absoluteValueOf{S} \cdot \absoluteValueOf{\ball(\theta')} + \absoluteValueOf{\boundaryOf_{\theta + \kappa + \theta'}^- F}. 
    \end{equation*}%
  \end{lemma}

  \begin{proof-sketch}
    For each $m \in F$, if $m \notin S^{+ \theta'}$, then $m \notin F^{-(\theta + \kappa + \theta')}$. Thus, $F \subseteq S^{+ \theta'} \cup F \smallsetminus F^{-(\theta + \kappa + \theta')} = \parens{\bigcup_{s \in S} \ball(s, \theta')} \cup \boundaryOf_{\theta + \kappa + \theta'}^- F$. Hence, $\absoluteValueOf{F} \leq \absoluteValueOf{S} \cdot \absoluteValueOf{\ball(\theta')} + \absoluteValueOf{\boundaryOf_{\theta + \kappa + \theta'}^- F}$.
  \end{proof-sketch}

  \begin{proof}
    Let $m \in F \smallsetminus \bigcup_{s \in S} \ball(s, \theta')$. Because $\family{\ball(t, \theta')}_{t \in T}$ is a cover of $M$, there is a $t \in T$ such that $m \in \ball(t, \theta')$. Because $m \notin \bigcup_{s \in S} \ball(s, \theta')$, we have $t \notin S$ and hence $\ball(t, \theta)^{+\kappa} \nsubseteq F$. Because $m \in \ball(t, \theta')$, we have $d(m, t) \leq \theta'$ and hence $t \in \ball(m, \theta') = \setOf{m}^{+\theta'}$.
    Suppose that $m \in F^{-(\theta + \kappa + \theta')}$. Then, according to \cite[Item~2 of Lemma~6.2]{wacker:growth:2017} and \cite[Item~5 of Lemma~6.4]{wacker:growth:2017}, 
    \begin{align*} 
      \ball(t, \theta)^{+\kappa}
      &= (\setOf{t}^{+\theta})^{+\kappa}\\
      &\subseteq \parens[\Big]{\parens[\big]{\setOf{m}^{+\theta'}}^{+\theta}}^{+\kappa}\\
      &= \setOf{m}^{+(\theta + \kappa + \theta')}\\
      &\subseteq (F^{-(\theta + \kappa + \theta')})^{+(\theta + \kappa + \theta')}\\
      &= F,
    \end{align*}
    which contradicts that $\ball(t, \theta)^{+\kappa} \nsubseteq F$. Hence, $m \in F \smallsetminus F^{-(\theta + \kappa + \theta')} = \boundaryOf_{\theta + \kappa + \theta'}^- F$. Therefore,
    \begin{equation*}
      F \subseteq \parens[\big]{\bigcup_{s \in S} \ball(s, \theta')} \cup \boundaryOf_{\theta + \kappa + \theta'}^- F.
    \end{equation*}
    Moreover, because $S \subseteq \bigcup_{s \in S} \ball(s, \theta)^{+\kappa} \subseteq F$ and $F$ is finite, the set $S$ is finite. 
    And, for each $s \in S$, according to \cite[Corollary~5.11]{wacker:growth:2017}, we have $\absoluteValueOf{\ball(s, \theta')} = \absoluteValueOf{\ball(\theta')}$. In conclusion,
    \begin{align*}
      \absoluteValueOf{F} &\leq \sum_{s \in S} \absoluteValueOf{\ball(\theta')} + \absoluteValueOf{\boundaryOf_{\theta + \kappa + \theta'}^- F}\\
              &=    \absoluteValueOf{S} \cdot \absoluteValueOf{\ball(\theta')} + \absoluteValueOf{\boundaryOf_{\theta + \kappa + \theta'}^- F}. 
    \end{align*}
  \end{proof}

  The number of elements that the components of a right Følner net share with a tiling is asymptotically bounded below away from zero, which is shown in

  \begin{corollary} 
  \label{cor:subshifts-upper-bound-of-tiling-cap-folner-net} 
    Let $\mathcal{R}$ be right amenable, let $\net{F_i}_{i \in I}$ be a right Følner net in $\mathcal{R}$, let $\theta$, $\kappa$, and $\theta'$ be three non-negative integers, let $T$ be a subset of $M$ such that $\family{\ball(t, \theta')}_{t \in T}$ is a cover of $M$. There is a positive real number $\varepsilon \in \R_{> 0}$ and there is an index $i_0 \in I$ such that, for each index $i \in I$ with $i \geq i_0$, we have $\absoluteValueOf{T \cap F_i^{-(\theta + \kappa)}} \geq \varepsilon \absoluteValueOf{F_i}$.
  \end{corollary} 

  \begin{proof}
    Let $i \in I$ and let $T_i = T \cap F_i^{-(\theta + \kappa)}$. According to \cref{lem:set-contained-in-unition-of-balls-and-of-interior-of-set},
    \begin{equation*}
      \absoluteValueOf{F_i} \leq \absoluteValueOf{T_i} \cdot \absoluteValueOf{\ball(\theta')} + \absoluteValueOf{\boundaryOf_{\theta + \kappa + \theta'}^- F_i}.
    \end{equation*}
    Thus, because $\boundaryOf_{\theta + \kappa + \theta'}^- F_i \subseteq \boundaryOf_{\theta + \kappa + \theta'} F_i$,
    \begin{equation*}
      \frac{\absoluteValueOf{T_i}}{\absoluteValueOf{F_i}} \geq \frac{1}{\absoluteValueOf{\ball(\theta')}} \cdot \parens*{1 - \frac{\absoluteValueOf{\boundaryOf_{\theta + \kappa + \theta'} F_i}}{\absoluteValueOf{F_i}}}.
    \end{equation*}
    According to \cite[Theorem~10.6]{wacker:growth:2017}, there is an $i_0 \in I$ such that
    \begin{equation*} 
      \ForEach i \in I \Holds \parens*{i \geq i_0 \implies \frac{\absoluteValueOf{\boundaryOf_{\theta + \kappa + \theta'} F_i}}{\absoluteValueOf{F_i}} \leq \frac{1}{2}}.
    \end{equation*}
    Let $\varepsilon = 1 / (2 \absoluteValueOf{\ball(\theta')})$. Then, for each $i \in I$ with $i \geq i_0$,
    \begin{equation*}
      \frac{\absoluteValueOf{T_i}}{\absoluteValueOf{F_i}} \geq \frac{1}{2 \absoluteValueOf{\ball(\theta')}} = \varepsilon. 
    \end{equation*}
  \end{proof}

  If a shift space has at least two points, then, for each non-empty domain, it has at least two patterns.

  \begin{lemma}
  \label{lem:subshift-with-at-least-two-points-has-at-least-two-patterns-for-each-domain}
    Let $X$ be a subshift of $Q^M$ such that $\absoluteValueOf{X} \geq 2$ and let $A$ be a non-empty subset of $M$. Then, $\absoluteValueOf{X_A} \geq 2$.
  \end{lemma}

  \begin{proof}
    Because $\absoluteValueOf{X} \geq 2$, there are $x$ and $x'$ in $X$ such that $x \neq x'$. Thus, there is an $m \in M$ such that $x(m) \neq x'(m)$. And, because $A$ is non-empty, there is an $a \in A$. The element $h = g_{m_0, a} g_{m_0, m}^{-1}$ is contained in $H$ and satisfies $h^{-1} \actsOnPoint a = m$. Hence, $(h \actsOnMap x)(a) \neq (h \actsOnMap x')(a)$. Therefore, $(h \actsOnMap x)\restrictedTo_A$ and $(h \actsOnMap x')\restrictedTo_A$ are distinct and are contained in $X_A$. In conclusion, $\absoluteValueOf{X_A} \geq 2$.
  \end{proof}

  A subset of a strongly irreducible shift space has less entropy than that space if about each point of a tiling the subset has fewer patterns of a certain radius than the space, which is shown in

  \begin{theorem} 
  \label{thm:entropy-less-if-real-subsets-for-rho-kappa-theta-tiling} 
    Let $\mathcal{R}$ be right amenable, let $\mathcal{F} = \net{F_i}_{i \in I}$ be a right Følner net in $\mathcal{R}$, let $X$ be a $\kappa$-strongly irreducible subshift of $Q^M$ such that $\absoluteValueOf{X} \geq 2$, let $Y$ be a subset of $X$, and let $T$ be a $\ntuple{\theta, \kappa, \theta'}$-tiling of $\mathcal{R}$ such that, for each element $t \in T$, we have $Y_{\ball(t, \theta)} \subsetneqq X_{\ball(t, \theta)}$. Then, $\entropyOf_{\mathcal{F}}(Y) < \entropyOf_{\mathcal{F}}(X)$.
  \end{theorem} 

  \begin{usage-note}
    In the proof, $\kappa$-strong irreducibility is used to apply \cref{lem:technical-inequality-for-theorem-entropy-less-if-real-subsets-for-rho-kappa-theta-tiling} yielding the inequality $\absoluteValueOf{X_{F_i} \smallsetminus \bigcup_{t \in T_i} \pi_{i, t}^{-1}(p_t)} \leq (1 - \xi^{-1})^{\absoluteValueOf{T_i}} \cdot \absoluteValueOf{X_{F_i}}$.
  \end{usage-note}

  \begin{proof-sketch} 
    Let $p_t \in X_{\ball(t, \theta)} \smallsetminus Y_{\ball(t, \theta)}$ and let $T_i = T \cap F_i^{-(\theta + \kappa)}$. Then, $Y_{F_i} \subseteq X_{F_i} \smallsetminus \bigcup_{t \in T_i} \pi_{i, t}^{-1}(p_t)$. Hence, $\absoluteValueOf{Y_{F_i}} \leq (1 - \xi^{-1})^{\absoluteValueOf{T_i}} \cdot \absoluteValueOf{X_{F_i}}$. Therefore, $\log\absoluteValueOf{Y_{F_i}}/\absoluteValueOf{F_i} \leq \log(1 - \xi^{-1}) \cdot \absoluteValueOf{T_i}/\absoluteValueOf{F_i} + \log\absoluteValueOf{X_{F_i}}/\absoluteValueOf{F_i}$. In conclusion, because $\log(1 - \xi^{-1}) < 0$ and $\net{\absoluteValueOf{T_i}/\absoluteValueOf{F_i}}_{i \in I}$ is eventually bounded below away from zero, we have $\entropyOf_{\mathcal{F}}(Y) < \entropyOf_{\mathcal{F}}(X)$.
  \end{proof-sketch}

  \begin{proof} 
    For each $t \in T$, because $Y_{\ball(t, \theta)} \subsetneqq X_{\ball(t, \theta)}$, we have $X_{\ball(t, \theta)} \smallsetminus Y_{\ball(t, \theta)} \neq \emptyset$. Let $\family{p_t}_{t \in T}$ be a transversal of $\family{X_{\ball(t, \theta)} \smallsetminus Y_{\ball(t, \theta)}}_{t \in T}$ and let $\xi = \absoluteValueOf{X_{\ball(\theta)^{+\kappa}}}$. Furthermore, let $i \in I$, let $T_i = T \cap F_i^{-(\theta + \kappa)}\ (= \setOf{t \in T \suchThat \ball(t, \theta)^{+\kappa} \subseteq F_i})$ and, for each $t \in T_i$, let $\pi_{i, t} \from X_{F_i} \to X_{\ball(t, \theta)}$, $p \mapsto p\restrictedTo_{\ball(t, \theta)}$. Note that, because $\absoluteValueOf{X} \geq 2$, according to \cref{lem:subshift-with-at-least-two-points-has-at-least-two-patterns-for-each-domain}, we have $\xi \geq 2$ and hence $1 - \xi^{-1} > 0$. 

    Because $\family{\ball(t, \theta)}_{t \in T}$ is pairwise at least $\kappa + 1$ apart, according to \cref{lem:technical-inequality-for-theorem-entropy-less-if-real-subsets-for-rho-kappa-theta-tiling},
    \begin{equation*}
      \absoluteValueOf{X_{F_i} \smallsetminus \bigcup_{t \in T_i} \pi_{t}^{-1}(p_{t})}
      \leq
      (1 - \xi^{-1})^{\absoluteValueOf{T_i}} \cdot \absoluteValueOf{X_{F_i}}.
    \end{equation*}
    For each $t \in T_i$, because $p_t \notin Y_{\ball(t, \theta)}$, we have $\pi_{i, t}^{-1}(p_t) \cap Y_{F_i} = \emptyset$. Hence, $\parens*{\bigcup_{t \in T_i} \pi_{i, t}^{-1}(p_t)} \cap Y_{F_i} = \emptyset$. Therefore,
    \begin{align*}
      \absoluteValueOf{Y_{F_i}} &=    \absoluteValueOf{Y_{F_i} \smallsetminus \bigcup_{t \in T_i} \pi_{i, t}^{-1}(p_t)}\\
                    &\leq \absoluteValueOf{X_{F_i} \smallsetminus \bigcup_{t \in T_i} \pi_{i, t}^{-1}(p_t)}\\
                    &\leq (1 - \xi^{-1})^{\absoluteValueOf{T_i}} \cdot \absoluteValueOf{X_{F_i}}.
    \end{align*}
    Thus,
    \begin{equation*}
      \frac{\log\absoluteValueOf{Y_{F_i}}}{\absoluteValueOf{F_i}}
      \leq \frac{\absoluteValueOf{T_i}}{\absoluteValueOf{F_i}} \cdot \log(1 - \xi^{-1})
           + \frac{\log\absoluteValueOf{X_{F_i}}}{\absoluteValueOf{F_i}}.
    \end{equation*}
    Because $\family{\ball(t, \theta')}_{t \in T}$ is a cover of $M$, according to \cref{cor:subshifts-upper-bound-of-tiling-cap-folner-net}, there is an $\varepsilon \in \R_{> 0}$ and there is an $i_0 \in I$ such that
    \begin{equation*}
      \ForEach i \in I \Holds \parens*{i \geq i_0 \implies \frac{\absoluteValueOf{T_i}}{\absoluteValueOf{F_i}} \geq \varepsilon}.
    \end{equation*}
    Hence, because $\log(1 - \xi^{-1}) < 0$,
    \begin{equation*}
      \entropyOf_{\mathcal{F}}(Y) \leq \varepsilon \cdot \log(1 - \xi^{-1}) + \entropyOf_{\mathcal{F}}(X)
                                <    \entropyOf_{\mathcal{F}}(X). 
    \end{equation*}
  \end{proof}

  \section{The Garden of Eden Theorems} 
  \label{sec:gardens-of-eden}

  \subsubsection*{Contents.} The image of a local map to a strongly irreducible shift space that is not surjective does not have maximal entropy (see \cref{thm:subshift-not-surjective-implies-less-entropy}). And the converse of that statement obviously holds. Moreover, a local map from a strongly irreducible shift space of finite type whose image has less entropy than its domain is not pre-injective (see \cref{thm:subshift-less-entropy-implies-not-pre-injective}). And the converse of that statement also holds (see \cref{thm:subshift-not-pre-injective-implies-less-entropy}). These four statements establish the Garden of Eden theorem (see Main Theorem~\ref{thm:subshift-garden-of-eden}). It follows that strongly irreducible shift spaces of finite type have the Moore and the Myhill property (see \cref{cor:Moore-and-Myhill}). 


  \subsubsection*{Body.} Because a local map that is not surjective has a Garden of Eden pattern, the entropy of its image is not maximal, which is shown in

  \begin{theorem} 
  \label{thm:subshift-not-surjective-implies-less-entropy}
    Let $\mathcal{R}$ be right amenable, let $\mathcal{F}$ be a right Følner net in $\mathcal{R}$, let $M$ be infinite, let $X$ be a non-empty subshift of $Q^M$, let $Y$ be a strongly irreducible subshift of $Q^M$, and let $\Delta$ be a local map from $X$ to $Y$ that is not surjective. Then, $\entropyOf_{\mathcal{F}}(\Delta(X)) < \entropyOf_{\mathcal{F}}(Y)$.
  \end{theorem}

  \begin{usage-note}
    In the proof, infiniteness of $M$ is used to apply \cref{thm:there-is-a-rho-kappa-theta-tiling} yielding a tiling, locality of $\Delta$ is used to apply \cref{lem:image-of-local-map-is-subshift} yielding that $\Delta(X)$ is a subshift of $Q^M$, and strong irreducibility of $Y$ is used to apply \cref{thm:entropy-less-if-real-subsets-for-rho-kappa-theta-tiling} yielding a strict inequality for entropies.
  \end{usage-note}

  \begin{proof-sketch} 
    Because $\Delta$ is not surjective, there is a Garden of Eden configuration. Thus, because $\Delta$ is local, there is a Garden of Eden pattern. Hence, there are too many Garden of Eden configurations for the entropy to be maximal. 
  \end{proof-sketch}

  \begin{proof}
    Because $Y$ is strongly irreducible, there is a $\kappa \in \N_0$ such that $Y$ is $\kappa$-strongly irreducible. And, because $\Delta$ is not surjective, there is a $y \in Y \smallsetminus \Delta(X)$. Hence, according to \cref{lem:if-restrictions-to-balls-are-in-shift-then-so-is-pattern}, there is a $\rho \in \N_0$ such that $y\restrictedTo_{\ball(\rho)} \notin (\Delta(X))_{\ball(\rho)}$ and thus $y\restrictedTo_{\ball(\rho)} \in Y_{\ball(\rho)} \smallsetminus (\Delta(X))_{\ball(\rho)}$. And, because $M$ is infinite, according to \cref{thm:there-is-a-rho-kappa-theta-tiling}, there is a $\ntuple{\rho, \kappa, \theta'}$-tiling $T$ of $\mathcal{R}$.

    According to \cref{lem:image-of-local-map-is-subshift}, the set $\Delta(X)$ is a subshift of $Q^M$. And, for each $t \in T$, according to \cite[Corollary~5.14]{wacker:growth:2017}, we have $t \isSemiActedUponBy \ball(\rho) = \ball(t, \rho)$. Therefore, for each $t \in T$, because $t \actsByItsCoordinateOn \blank$ is bijective and according to \cref{rem:pattern-belongs-to-shift-if-and-only-if-translated-pattern-belongs}, we have $t \actsByItsCoordinateOn (y\restrictedTo_{\ball(\rho)}) \in Y_{\ball(t, \rho)} \smallsetminus (\Delta(X))_{\ball(t, \rho)}$ and thus $(\Delta(X))_{\ball(t, \rho)} \subsetneqq Y_{\ball(t, \rho)}$. 

    Because $X$ is non-empty and $\Delta$ is not surjective, we have $\absoluteValueOf{Y} \geq 2$. In conclusion, because $Y$ is $\kappa$-strongly irreducible, according to \cref{thm:entropy-less-if-real-subsets-for-rho-kappa-theta-tiling}, we have $\entropyOf_{\mathcal{F}}(\Delta(X)) < \entropyOf_{\mathcal{F}}(Y)$.
  \end{proof}

  If there are less patterns in the codomain of a local map than in its domain, at least two patterns have the same image, which is shown in

  \begin{lemma} 
  \label{lem:there-are-patterns-that-agree-on-boundary-and-have-the-same-image-under-Delta-restriction-k}
    Let $X$ be a $\kappa$-strongly irreducible subshift of $Q^M$, let $Y$ be a subshift of $Q^M$, let $\Delta$ be a $\kappa$-local map from $X$ to $Y$, and let $F$ be a finite subset of $M$ such that $\absoluteValueOf{Y_{F^{+2\kappa}}} < \absoluteValueOf{X_F}$. There are two patterns $p$ and $p'$ in $X_{F^{+3\kappa}}$ such that $p \neq p'$, $p\restrictedTo_{\boundaryOf_{2\kappa}^+ F^{+\kappa}} = p'\restrictedTo_{\boundaryOf_{2\kappa}^+ F^{+\kappa}}$, and $\Delta_{F^{+3\kappa}}^-(p) = \Delta_{F^{+3\kappa}}^-(p')$. 
  \end{lemma}

  \begin{usage-note}
    In the proof, strong irreducibility of $X$ is used to extend an in $X$ allowed $F$-pattern by an in $X$ allowed $\boundaryOf_{2\kappa}^+ F^{+\kappa}$-pattern and an $\boundaryOf_\kappa^+ F$-pattern to an in $X$ allowed $F^{+3\kappa}$-pattern; and $\kappa$-locality of $\Delta$ is used to restrict it to a map from $X_{F^{+3\kappa}}$ to $Y_{F^{+2\kappa}}$.
  \end{usage-note}

  \begin{proof} 
    Because $\absoluteValueOf{Y_{F^{+2\kappa}}} < \absoluteValueOf{X_F}$, we have $\absoluteValueOf{X_F} > 0$, thus $X_F \neq \emptyset$, and hence $X \neq \emptyset$. 
    Therefore, there is a $v \in X_{\boundaryOf_{2\kappa}^+ F^{+\kappa}}$. Let $P_v = \setOf{p \in X_{F^{+3\kappa}} \suchThat p\restrictedTo_{\domainOf(v)} = v}$. Note that, according to \cite[Item~4 of Lemma~6.4]{wacker:growth:2017}, we have $\boundaryOf_{2\kappa}^+ F^{+\kappa} = F^{+3\kappa} \smallsetminus F^{+\kappa}$.

      Let $u \in X_F$. According to \cite[Corollary~6.6]{wacker:growth:2017}, we have $d(F, \domainOf(v)) \geq \kappa + 1$. Hence, because $X$ is $\kappa$-strongly irreducible, there is an $x \in X$ such that $x\restrictedTo_F = u$ and $x\restrictedTo_{\domainOf(v)} = v$. Let $p = x\restrictedTo_{F^{+3\kappa}}$. Then, $p\restrictedTo_F = u$ and $p \in P_v$.

    Therefore, $\absoluteValueOf{P_v} \geq \absoluteValueOf{X_F} > \absoluteValueOf{Y_{F^{+2\kappa}}}$. The restriction $\phi$ of $\Delta_{F^{+3\kappa}}^- \from X_{F^{+3\kappa}} \to Y_{F^{+2\kappa}}$ to $P_v \to \Delta_{F^{+3\kappa}}^-(P_v)$ is surjective. Note that, because $\Delta$ is $\kappa$-local and, according to \cite[Item~5 of Lemma~6.4]{wacker:growth:2017}, we have $(F^{+3\kappa})^{-\kappa} \supseteq F^{+2\kappa}$, we can choose $Y_{F^{+2\kappa}}$ as the codomain of $\Delta_{F^{+3\kappa}}^-$. If $\phi$ were injective, then $\absoluteValueOf{P_v} = \absoluteValueOf{\phi(P_v)} \leq \absoluteValueOf{Y_{F^{+2\kappa}}}$, which contradicts that $\absoluteValueOf{P_v} > \absoluteValueOf{Y_{F^{+2\kappa}}}$. Hence, $\phi$ is not injective. In conclusion, there are $p$, $p' \in P_v$ such that $p \neq p'$ and $\phi(p) = \phi(p')$.
  \end{proof}

  Because a local map, that has an image whose entropy is less than the entropy of its domain, maps at least two finite patterns to the same pattern, it is not pre-injective, which is shown in

  \begin{theorem} 
  \label{thm:subshift-less-entropy-implies-not-pre-injective}
    Let $\mathcal{R}$ be right amenable, let $\mathcal{F} = \net{F_i}_{i \in I}$ be a right Følner net in $\mathcal{R}$, let $X$ be a strongly irreducible subshift of $Q^M$ of finite type, let $Y$ be a subshift of $Q^M$, and let $\Delta$ be a local map from $X$ to $Y$ such that $\entropyOf_{\mathcal{F}}(\Delta(X)) < \entropyOf_{\mathcal{F}}(X)$. The map $\Delta$ is not pre-injective.
  \end{theorem}

  \begin{usage-note}
    In the proof, strong irreducibility of $X$ and locality of $\Delta$ are used to apply \cref{lem:there-are-patterns-that-agree-on-boundary-and-have-the-same-image-under-Delta-restriction-k} yielding two distinct finite patterns with the same domain, identical boundaries, and identical images; and of finite typeness of $X$ is used to apply \cref{cor:overlapping-patterns-can-be-glued} to identically extend these patterns to points of $X$.
  \end{usage-note}

  \begin{proof-sketch}
    Because the entropy of $\Delta(X)$ is less than the one of $X$, the number of finite patterns in $\Delta(X)$ grows slower than in $X$. Hence, there are two distinct finite patterns in $X$ that have the same image and these can be identically extended to two distinct points of $X$ that have the same image. Therefore, the map $\Delta$ is not pre-injective.
  \end{proof-sketch}

  \begin{proof} 
    According to \cref{rem:k-strongly-irreducible-for-greater-k}, \cref{lem:of-finite-type-implies-step} and \cref{rem:k-step-for-greater-k}, and \cref{rem:k-local-for-greater-k}, there is a $\kappa \in \N_0$ such that $X$ is $\kappa$-strongly irreducible, $X$ is $\kappa$-step, and $\Delta$ is $\kappa$-local.

    Let $Y = \Delta(X)$. According to \cref{rem:local-map-iff-cellular-automaton} and \cite[Lemma~11]{wacker:garden:2016} and the precondition $\entropyOf_{\mathcal{F}}(Y) < \entropyOf_{\mathcal{F}}(X)$, we have $\entropyOf_{\net{F_i^{+2\kappa}}_{i \in I}}(Y) \leq \entropyOf_{\mathcal{F}}(Y) < \entropyOf_{\mathcal{F}}(X)$. Hence, there is an $i \in I$ such that
    \begin{equation*}
      \frac{\log\absoluteValueOf{Y_{F_i^{+2\kappa}}}}{\absoluteValueOf{F_i}} < \frac{\log\absoluteValueOf{X_{F_i}}}{\absoluteValueOf{F_i}}.
    \end{equation*}
    Thus, $\log\absoluteValueOf{Y_{F_i^{+2\kappa}}} < \log\absoluteValueOf{X_{F_i}}$ and thus $\absoluteValueOf{Y_{F_i^{+2\kappa}}} < \absoluteValueOf{X_{F_i}}$.
    Therefore, because $X$ is $\kappa$-strongly irreducible and $\Delta$ is $\kappa$-local, according to \cref{lem:there-are-patterns-that-agree-on-boundary-and-have-the-same-image-under-Delta-restriction-k}, there are $p$ and $p'$ in $X_{F_i^{+3\kappa}}$ such that $p \neq p'$, $p\restrictedTo_{\boundaryOf_{2\kappa}^+ F_i^{+\kappa}} = p'\restrictedTo_{\boundaryOf_{2\kappa}^+ F_i^{+\kappa}}$, and $\Delta_{F_i^{+3\kappa}}^-(p) = \Delta_{F_i^{+3\kappa}}^-(p')$.

    Hence, because $X$ is $\kappa$-step, according to \cref{cor:overlapping-patterns-can-be-glued}, there are $x$ and $x'$ in $X$ such that $x\restrictedTo_{\domainOf(p)} = p$, $x'\restrictedTo_{\domainOf(p')} = p'$, and $x\restrictedTo_{M \smallsetminus F_i^{+\kappa}} = x'\restrictedTo_{M \smallsetminus F_i^{+\kappa}}$. In particular, because $p \neq p'$, we have $x \neq x'$ and, because $F_i^{+\kappa}$ is finite, the set $\differenceOf(x, x')$ is finite.

    Moreover, $\Delta(x)\restrictedTo_{F_i^{+2\kappa}} = \Delta_{F_i^{+3\kappa}}^-(p) = \Delta_{F_i^{+3\kappa}}^-(p') = \Delta(x')\restrictedTo_{F_i^{+2\kappa}}$. And, according to \cref{rem:local-map-iff-cellular-automaton} and \cite[Lemma~2]{wacker:garden:2016}, we have $\Delta(x)\restrictedTo_{M \smallsetminus F_i^{+2\kappa}} = \Delta(x')\restrictedTo_{M \smallsetminus F_i^{+2\kappa}}$. Therefore, $\Delta(x) = \Delta(x')$. In conclusion, $\Delta$ is not pre-injective.
  \end{proof}

  If in a point of a shift space we replace all occurrences of a pattern by another pattern with the same image that agree on a big enough boundary, we get a new point of the shift space in which the first pattern does not occur that has the same image as the original point, which is shown in

  \begin{lemma} 
  \label{lem:exchanging-pattern-by-other-pattern-with-same-image-yields-configuration-with-same-image-and-we-stay-in-subshift}
    Let $X$ be a $\kappa$-step subshift of $Q^M$, let $Y$ be a subshift of $Q^M$, let $\Delta$ be a $\kappa$-local map from $X$ to $Y$, let $A$ be a subset of $M$, let $p$ and $p'$ be two patterns in $X_{A^{+2\kappa}}$ such that $p\restrictedTo_{\boundaryOf_{2\kappa}^+ A} = p'\restrictedTo_{\boundaryOf_{2\kappa}^+ A}$ and $\Delta_{A^{+2\kappa}}^-(p) = \Delta_{A^{+2\kappa}}^-(p')$. Furthermore, let $c$ be a point of $X$ and let $T$ be a subset of $M$ such that the family $\family{t \isSemiActedUponBy A^{+2\kappa}}_{t \in T}$ is pairwise disjoint and that, for each element $t \in T$, we have $p \occursIn_t c$. Put 
    \begin{equation*} 
      c' = c\restrictedTo_{M \smallsetminus (\bigcup_{t \in T} t \isSemiActedUponBy A^{+2\kappa})} \times \coprod_{t \in T} t \actsByItsCoordinateOn p'.
    \end{equation*}
    Then, for each element $t \in T$, we have $p' \occursIn_t c'$, and $c' \in X$, and $\Delta(c) = \Delta(c')$. In particular, if $p \neq p'$, then, for each element $t \in T$, we have $p \not\occursIn_t c'$.
  \end{lemma}

  \begin{usage-note}
    In the proof, $\kappa$-stepness of $X$ is used to apply \cref{lem:overlapping-global-configurations-can-be-glued} to deduce that $c'$ is a point of $X$; and locality of $\Delta$ is used to deduce that $\Delta(c) = \Delta(c')$.
  \end{usage-note}

  \begin{proof}
    There are $x$ and $x'$ in $X$ such that $x\restrictedTo_{A^{+2\kappa}} = p$ and $x'\restrictedTo_{A^{+2\kappa}} = p'$. Thus, for each $t \in T$, we have $(t \actsByItsCoordinateOn x')\restrictedTo_{t \isSemiActedUponBy A^{+2\kappa}} = t \actsByItsCoordinateOn p'$. Hence,
    \begin{align*}
      c' = c\restrictedTo_{M \smallsetminus (\bigcup_{t \in T} t \isSemiActedUponBy A^{+2\kappa})} \times \coprod_{t \in T} (t \actsByItsCoordinateOn x')\restrictedTo_{t \isSemiActedUponBy A^{+2\kappa}}
    \end{align*}
    Moreover, for each $t \in T$, according to \cref{rem:shift-invariance-induces-invariance-under-right-semi-action}, we have $t \actsByItsCoordinateOn x' \in X$. And, by precondition, $\family{(t \isSemiActedUponBy A)^{+2\kappa}}_{t \in T}$ is pairwise disjoint (where we used that $t \isSemiActedUponBy A^{+2\kappa} = (t \isSemiActedUponBy A)^{+2\kappa}$, which holds according to \cite[Item~9 of Lemma~1]{wacker:garden:2016}). And, for each $t \in T$, we have $(t \actsByItsCoordinateOn x')\restrictedTo_{\boundaryOf_{2\kappa}^+(t \isSemiActedUponBy A)} = (t \actsByItsCoordinateOn p')\restrictedTo_{\boundaryOf_{2\kappa}^+(t \isSemiActedUponBy A)} = (t \actsByItsCoordinateOn p)\restrictedTo_{\boundaryOf_{2\kappa}^+(t \isSemiActedUponBy A)} = c\restrictedTo_{\boundaryOf_{2\kappa}^+(t \isSemiActedUponBy A)}$ (where we used that $\boundaryOf_{2\kappa}^+(t \isSemiActedUponBy A) = t \isSemiActedUponBy \boundaryOf_{2\kappa}^+ A$, which holds according to \cite[Item~9 of Lemma~1]{wacker:garden:2016}). Therefore, because $X$ is $\kappa$-step, according to \cref{lem:overlapping-global-configurations-can-be-glued}, we have $c' \in X$.

    Let $m \in M$.
    \begin{description}
      \item[Case 1] $\Exists t \in T \SuchThat m \in t \isSemiActedUponBy A^{+\kappa}$. Then, according to \cite[Item~2 of Lemma~6.2]{wacker:growth:2017} and \cite[Item~3 of Lemma~6.4]{wacker:growth:2017}
                     \begin{align*}
                       m \isSemiActedUponBy \ball(\kappa)
                       &\subseteq (t \isSemiActedUponBy A^{+\kappa}) \isSemiActedUponBy \ball(\kappa)\\
                       &=         (t \isSemiActedUponBy A^{+\kappa})^{+\kappa}\\
                       &=          t \isSemiActedUponBy A^{+2\kappa}.
                     \end{align*}
                     Hence, because $\Delta$ is $\kappa$-local, 
                     \begin{align*}
                       \Delta(c')(m)
                       &= \Delta_{t \isSemiActedUponBy A^{+2\kappa}}^-(t \actsByItsCoordinateOn p')\\
                       &= t \actsByItsCoordinateOn \Delta_{A^{+2\kappa}}^-(p')\\
                       &= t \actsByItsCoordinateOn \Delta_{A^{+2\kappa}}^-(p)\\
                       &= \Delta_{t \isSemiActedUponBy A^{+2\kappa}}^-(t \actsByItsCoordinateOn p)\\
                       &= \Delta(c)(m).
                     \end{align*}
      \item[Case 2] $\ForEach t \in T \Holds m \notin t \isSemiActedUponBy A^{+\kappa}$. Then, $m \in M \smallsetminus \bigcup_{t \in T} t \isSemiActedUponBy A^{+\kappa}$. Hence, according to \cite[Item~2 of Lemma~6.2]{wacker:growth:2017}, \cite[Item~3 of Lemma~1]{wacker:garden:2016},
                      and \cite[Item~5 of Lemma~6.4]{wacker:growth:2017},
                     \begin{align*}
                       m \isSemiActedUponBy \ball(\kappa)
                       &\subseteq \parens[\big]{M \smallsetminus \bigcup_{t \in T} t \isSemiActedUponBy A^{+\kappa}} \isSemiActedUponBy \ball(\kappa)\\
                       &=         \parens[\big]{M \smallsetminus \bigcup_{t \in T} t \isSemiActedUponBy A^{+\kappa}}^{+\kappa}\\
                       &=          M \smallsetminus \parens[\big]{\bigcup_{t \in T} t \isSemiActedUponBy A^{+\kappa}}^{-\kappa}\\
                       &\subseteq  M \smallsetminus \bigcup_{t \in T} (t \isSemiActedUponBy A^{+\kappa})^{-\kappa}\\
                       &\subseteq  M \smallsetminus \bigcup_{t \in T} t \isSemiActedUponBy A^{+ 0}\\
                       &=          M \smallsetminus \bigcup_{t \in T} t \isSemiActedUponBy A.
                     \end{align*}
                     Therefore, because $\Delta$ is $\kappa$-local and $c'\restrictedTo_{M \smallsetminus \bigcup_{t \in T} t \isSemiActedUponBy A} = c\restrictedTo_{M \smallsetminus \bigcup_{t \in T} t \isSemiActedUponBy A}$, we have $\Delta(c')(m) = \Delta(c)(m)$.
    \end{description}
    In either case, $\Delta(c')(m) = \Delta(c)(m)$. Therefore, $\Delta(c') = \Delta(c)$.
  \end{proof}

  Because a local map that is not pre-injective maps at least two finite patterns to the same pattern, the entropy of its image is less than the entropy of its domain, which is shown in

  \begin{theorem} 
  \label{thm:subshift-not-pre-injective-implies-less-entropy}
    Let $M$ be infinite, let $X$ be a strongly irreducible subshift of $Q^M$ of finite type, let $Y$ be a subshift of $Q^M$, and let $\Delta$ be a local map from $X$ to $Y$ that is not pre-injective. Then, $\entropyOf_{\mathcal{F}}(\Delta(X)) < \entropyOf_{\mathcal{F}}(X)$.
  \end{theorem}

  \begin{usage-note}
    In the proof, infiniteness of $M$ is used to apply \cref{thm:there-is-a-rho-kappa-theta-tiling} yielding a tiling, strong irreducibility of $X$ is used to apply \cref{thm:entropy-less-if-real-subsets-for-rho-kappa-theta-tiling} yielding a strict inequality for entropies, finite typeness of $X$ and locality of $\Delta$ is used to apply \cref{lem:exchanging-pattern-by-other-pattern-with-same-image-yields-configuration-with-same-image-and-we-stay-in-subshift} yielding that the image of all points of $X$ in which a certain pattern does not occur at points of a tiling is the same as the image of $X$.
  \end{usage-note}

  \begin{proof-sketch}
    Because $\Delta$ is not pre-injective, there are two distinct points of $X$ with the same image that differ only in finitely many cells. Thus, there are two distinct finite patterns, say $p$ and $p'$, with the same image. Hence, the image of $X$ is equal to the image of the set $Z$ of all points of $X$ in which the pattern $p$ does not occur. Because there are too many points not in $Z$, this set does have less entropy than $X$.
  \end{proof-sketch}

  \begin{proof}
    According to \cref{rem:k-strongly-irreducible-for-greater-k}, \cref{lem:of-finite-type-implies-step} and \cref{rem:k-step-for-greater-k}, and \cref{rem:k-local-for-greater-k}, there is a $\kappa \in \N_0$ such that $X$ is $\kappa$-strongly irreducible, $X$ is $\kappa$-step, and $\Delta$ is $\kappa$-local.

    Because $\Delta$ is not pre-injective, there are $c$ and $c'$ in $X$ such that $\differenceOf(c, c')$ is finite, $\Delta(c) = \Delta(c')$, and $c \neq c'$; in particular, $\absoluteValueOf{X} \geq 2$. Hence, there is a $\rho \in \N_0$ such that $\differenceOf(c, c') \subseteq \ball(\rho)$. Let $p = c\restrictedTo_{\ball(\rho)^{+2\kappa}}$ and let $p' = c'\restrictedTo_{\ball(\rho)^{+2\kappa}}$. Then, $p \neq p'$; $m_0 \in \ball(\rho)$; $p$, $p' \in X_{\ball(\rho)^{+2\kappa}}$; $p\restrictedTo_{\boundaryOf_{2\kappa}^+ \ball(\rho)} = p'\restrictedTo_{\boundaryOf_{2\kappa}^+ \ball(\rho)}$; and, because $\Delta(c) = \Delta(c')$, we have $\Delta_{\ball(\rho)^{+2\kappa}}^-(p) = \Delta_{\ball(\rho)^{+2\kappa}}^-(p')$.

    Because $M$ is infinite, according to \cref{thm:there-is-a-rho-kappa-theta-tiling}, there is a $\ntuple{\rho + 2 \kappa, \kappa, \theta'}$-tiling $T$ of $\mathcal{R}$. Let
    \begin{equation*}
      Z = \setOf{x \in X \suchThat \ForEach t \in T \Holds p \not\occursIn_t x}.
    \end{equation*}
    For each $t \in T$, according to \cref{rem:pattern-belongs-to-shift-if-and-only-if-translated-pattern-belongs}, we have $t \actsByItsCoordinateOn p \in X_{t \isSemiActedUponBy \ball(\rho)^{+2\kappa}} \smallsetminus Z_{t \isSemiActedUponBy \ball(\rho)^{+2\kappa}}$ and hence $Z_{t \isSemiActedUponBy \ball(\rho)^{+2\kappa}} \subsetneqq X_{t \isSemiActedUponBy \ball(\rho)^{+2\kappa}}$. Moreover, for each $t \in T$, according to \cite[Item~2 of Corollary~6.3]{wacker:growth:2017} and \cite[Corollary~5.14]{wacker:growth:2017}, we have $t \isSemiActedUponBy \ball(\rho)^{+2\kappa} = \ball(t, \rho + 2 \kappa)$.
    Therefore, because $X$ is $\kappa$-strongly irreducible and $\absoluteValueOf{X} \geq 2$, according to \cref{thm:entropy-less-if-real-subsets-for-rho-kappa-theta-tiling}, we have $\entropyOf_{\mathcal{F}}(Z) < \entropyOf_{\mathcal{F}}(X)$. Hence, according to \cite[Theorem~3]{wacker:garden:2016}, we have $\entropyOf_{\mathcal{F}}(\Delta(Z)) < \entropyOf_{\mathcal{F}}(X)$.

    Let $x \in X$. Put $U = \setOf{t \in T \suchThat p \occursIn_t x}$. Because $X$ is $\kappa$-step and $\Delta$ is $\kappa$-local, according to \cref{lem:exchanging-pattern-by-other-pattern-with-same-image-yields-configuration-with-same-image-and-we-stay-in-subshift}, there is an $x' \in X$ such that $x' \in Z$ and $\Delta(x) = \Delta(x')$. Therefore, $\Delta(X) = \Delta(Z)$. In conclusion, $\entropyOf_{\mathcal{F}}(\Delta(X)) < \entropyOf_{\mathcal{F}}(X)$. 
  \end{proof}

  Because a right Følner net in a finite cell space is eventually equal to the set of cells, the entropy of a subset of the full shift is a function of the cardinality of that set, which is shown in

  \begin{lemma} 
  \label{lem:if-M-is-finite-abs-X-determined-by-entropy}
    Let $\mathcal{R}$ be right amenable, let $\mathcal{F} = \net{F_i}_{i \in I}$ be a right Følner net in $\mathcal{R}$, let $M$ be finite, and let $X$ be a subset of $Q^M$. Then,
    \begin{equation*}
      \absoluteValueOf{X} = \exp\parens[\big]{\absoluteValueOf{M} \cdot \entropyOf_{\mathcal{F}}(X)}.
    \end{equation*}
  \end{lemma}

  \begin{proof} 
    Let $F$ be a non-empty and finite subset of $M$ such that $F \neq M$. Then, because $\isSemiActedUponBy$ is transitive, there is a $\mathfrak{g} \in G \modulo G_0$ such that $(F \isSemiActedUponBy \mathfrak{g}) \cap (M \smallsetminus F) \neq \emptyset$. Hence, $F \isSemiActedUponBy \mathfrak{g} \nsubseteq F$, thus $F \nsubseteq (\blank \isSemiActedUponBy \mathfrak{g})^{-1}(F)$, thus $F \smallsetminus (\blank \isSemiActedUponBy \mathfrak{g})^{-1}(F) \neq \emptyset$, and therefore $\absoluteValueOf{F \smallsetminus (\blank \isSemiActedUponBy \mathfrak{g})^{-1}(F)} \neq 0$. On the other hand, $\absoluteValueOf{M \smallsetminus (\blank \isSemiActedUponBy \mathfrak{g})^{-1}(M)} = 0$. Moreover, because $M$ is finite, the set $\setOf{F \subseteq M \suchThat F \neq \emptyset, F \text{ finite}}$ is finite and hence its subset $\setOf{F_i \suchThat i \in I}$ is finite too. Furthermore, because $\mathcal{F}$ is a right Følner net,
    \begin{equation*}
      \ForEach \mathfrak{g} \in G \modulo G_0 \Holds \lim_{i \in I} \frac{\absoluteValueOf{F_i \smallsetminus (\blank \isSemiActedUponBy \mathfrak{g})^{-1}(F_i)}}{\absoluteValueOf{F_i}} = 0.
    \end{equation*}
    Altogether, $\mathcal{F}$ is eventually equal to $M$. Therefore,
    \begin{equation*}
      \entropyOf_{\mathcal{F}}(X)
      = \frac{\log\absoluteValueOf{\pi_M(X)}}{\absoluteValueOf{M}}
      = \frac{\log\absoluteValueOf{X}}{\absoluteValueOf{M}}.
    \end{equation*}
    In conclusion, $\absoluteValueOf{X} = \exp(\absoluteValueOf{M} \cdot \entropyOf_{\mathcal{F}}(X))$.
  \end{proof}

  Because surjectivity as well as pre-injectivity of a local map is characterised by maximal entropy of its image with respect to its codomain or domain, if both domains have the same entropy, then a local map is surjective if and only if it is pre-injective, which is shown in

  \begin{main-theorem}[Garden of Eden theorem] 
  \label{thm:subshift-garden-of-eden}
    Let $\mathcal{R}$ be right amenable, let $\mathcal{F}$ be a right Følner net in $\mathcal{R}$, let $X$ be a non-empty strongly irreducible subshift of $Q^M$ of finite type, let $Y$ be a strongly irreducible subshift of $Q^M$ such that $\entropyOf_{\mathcal{F}}(X) = \entropyOf_{\mathcal{F}}(Y)$, and let $\Delta$ be a local map from $X$ to $Y$. The map $\Delta$ is surjective if and only if it is pre-injective.
  \end{main-theorem}

  \begin{usage-note}
    In the proof, non-emptiness of $X$, strong irreducibility of $Y$, and locality of $\Delta$ are used to apply \cref{thm:subshift-not-surjective-implies-less-entropy} yielding a characterisation of surjectivity; and strong irreducibility and finite typeness of $X$, and locality of $\Delta$ are used to apply \cref{thm:subshift-less-entropy-implies-not-pre-injective,thm:subshift-not-pre-injective-implies-less-entropy} yielding a characterisation of pre-injectivity.
  \end{usage-note}

  \begin{proof}
    First, let $M$ be finite. Then, $Q^M$ is finite, and thus $X$ and $Y$ are finite. Hence, because $\entropyOf_{\mathcal{F}}(X) = \entropyOf_{\mathcal{F}}(Y)$, according to \cref{lem:if-M-is-finite-abs-X-determined-by-entropy}, we have $\absoluteValueOf{X} = \absoluteValueOf{Y}$. Therefore, $\Delta$ is surjective if and only if it is injective. Moreover, because $M$ is finite, the map $\Delta$ is pre-injective if and only if it is injective. In conclusion, $\Delta$ is surjective if and only if it is pre-injective.

    Secondly, let $M$ be infinite. According to \cref{thm:subshift-not-surjective-implies-less-entropy}, the map $\Delta$ is not surjective if and only if $\entropyOf_{\mathcal{F}}(\Delta(X)) < \entropyOf_{\mathcal{F}}(Y)$. And, according to \cref{thm:subshift-less-entropy-implies-not-pre-injective} and \cref{thm:subshift-not-pre-injective-implies-less-entropy}, because $\entropyOf_{\mathcal{F}}(X) = \entropyOf_{\mathcal{F}}(Y)$, we have $\entropyOf_{\mathcal{F}}(\Delta(X)) < \entropyOf_{\mathcal{F}}(Y)$ if and only if $\Delta$ is not pre-injective. Hence, $\Delta$ is not surjective if and only if it is not pre-injective. In conclusion, $\Delta$ is surjective if and only if it is pre-injective.
  \end{proof}

  \begin{corollary} 
  \label{cor:Moore-and-Myhill} 
    Let $\mathcal{R}$ be right amenable. Each strongly irreducible subshift of $Q^M$ of finite type has the Moore and the Myhill property. \qed
  \end{corollary} 

  \begin{corollary}
  \label{cor:Moore-and-Myhill-for-left-homogeneous-spaces}
    Let $\mathcal{M} = \ntuple{M, G, \actsOnPoint}$ be a right amenable and finitely right generated left homogeneous space with finite stabilisers and let $Q$ be a finite set. For each coordinate system $\mathcal{K}$ for $\mathcal{M}$ and each $\mathcal{K}$-big subgroup $H$ of $G$, each strongly irreducible subshift of $Q^M$ of finite type with respect to $\mathcal{R} = \ntuple{\mathcal{M}, \mathcal{K}}$ and $H$ has the Moore and the Myhill property with respect to $\mathcal{R}$ and $H$. \qed
  \end{corollary}

  \begin{remark}
    Note that in \cref{cor:Moore-and-Myhill-for-left-homogeneous-spaces} we do not have to choose a finite and symmetric right generating set $S$, because being a subshift, being strongly irreducible, being of finite type, being local, being surjective, being pre-injective, having the Moore property, and having the Myhill property, do not depend on the choice of a finite and symmetric right generating set of $\mathcal{R}$; the reason for the properties that depend on the metric induced by such a right generating set is that those metrics are, according to \cite[Corollary 8.7]{wacker:growth:2017}, pairwise Lipschitz equivalent.
  \end{remark}

  \begin{example}[From Golden Mean to Even Shift]
    In the situation of \cref{ex:from-golden-mean-to-even-shift-local-map}, let $\mathcal{F}$ be the right Følner net $\sequence{\setOf{1, \dotsc, n}}_{n \in \N_+}$. Recall that the golden mean shift $X$ is non-empty (\cref{ex:shift:golden-mean}), strongly irreducible (\cref{ex:strongly-irreducible}), and of finite type (\cref{ex:of-finite-type-or-not}); and that the even shift $Y$ is strongly irreducible (\cref{ex:strongly-irreducible}) but \emph{not} of finite type (\cref{ex:of-finite-type-or-not}). And, according to \cite[Example 4.1.4]{lind:marcus:1995} and \cite[Example 4.1.6]{lind:marcus:1995}, the entropy of $X$ with respect to $\mathcal{F}$ and the one of $Y$ are both the golden mean $(1 + \sqrt{5})/2$. Therefore, according to \cref{thm:subshift-garden-of-eden}, because the local map $\Delta$ from $X$ to $Y$ is surjective (\cref{ex:from-golden-mean-to-even-shift-local-map}), it is also pre-injective. However, it is not injective, because the two points of $X$ with alternating $0$'s and $1$'s, that is, those of the form $\dotso 010101 \dotso$, are both mapped to the point of $Y$ with only $0$'s, that is, the one of the form $\dotso 000 \dotso$. 
  \end{example}

\end{document}